\documentclass[12pt,oneside,english,twoside, reqno]{amsart}
\usepackage[T1]{fontenc}
\usepackage[latin9]{luainputenc}
\usepackage{geometry}
\usepackage{amstext}
\usepackage{amsthm}
\usepackage{amssymb}

\makeatletter
\numberwithin{equation}{section}
\numberwithin{figure}{section}
\theoremstyle{plain}
\newtheorem{thm}{\protect\theoremname}[section]
\theoremstyle{definition}
\newtheorem{defn}[thm]{\protect\definitionname}
\theoremstyle{remark}
\newtheorem*{rem*}{\protect\remarkname}
\theoremstyle{plain}
\newtheorem{prop}[thm]{\protect\propositionname}
\theoremstyle{plain}
\newtheorem{lem}[thm]{\protect\lemmaname}

\usepackage{versions}

\usepackage{hyperref}

\newcommand{\add}{{
  \bigskip
  \footnotesize
  \textsc{Jarkko Siltakoski,  Department of Mathematics and Statistics, P.O.Box 35, FIN-40014, University of Jyväskylä, Finland}\par\nopagebreak
  \textit{E-mail address}: jarkko.siltakoski@jyu.fi
}}

\makeatletter
\@namedef{subjclassname@2020}{%
  \textup{2020} Mathematics Subject Classification}
\makeatother

\makeatother

\usepackage{babel}
\providecommand{\definitionname}{Definition}
\providecommand{\lemmaname}{Lemma}
\providecommand{\propositionname}{Proposition}
\providecommand{\remarkname}{Remark}
\providecommand{\theoremname}{Theorem}

\begin{document}
\global\long\def\d{\,d}%
\global\long\def\tr{\mathrm{tr}}%
\global\long\def\supp{\operatorname{spt}}%
\global\long\def\div{\operatorname{div}}%
\global\long\def\osc{\operatorname{osc}}%
\global\long\def\essup{\esssup}%
\global\long\def\aint{\dashint}%
\global\long\def\essinf{\esssinf}%
\global\long\def\essliminf{\esssliminf}%
\global\long\def\sa{|}%
\global\long\def\sgn{\operatorname{sgn}}%
 
\global\long\def\dist{\operatorname{dist}}%
\global\long\def\diam{\operatorname{diam}}%
\global\long\def\pmin{p_{\text{min}}}%
\global\long\def\pmax{p_{\text{max}}}%
\excludeversion{details}\date{October 2021}
\title[hölder gradient regularity]{Hölder gradient regularity for the inhomogeneous normalized $p(x)$-Laplace
equation}
\begin{abstract}
We prove the local gradient Hölder regularity of viscosity solutions
to the inhomogeneous normalized $p(x)$-Laplace equation
\[
-\Delta_{p(x)}^{N}u=f(x),
\]
where $p$ is Lipschitz continuous, $\inf p>1$, and $f$ is continuous
and bounded. 
\end{abstract}

\author{Jarkko Siltakoski}
\maketitle

\section{Introduction}

We study the \textit{inhomogeneous normalized $p(x)$-Laplace equation}

\begin{equation}
-\Delta_{p(x)}^{N}u=f(x)\quad\text{in }B_{1},\label{eq:normalized p(x)}
\end{equation}
where
\[
-\Delta_{p(x)}^{N}u:=-\Delta u-(p(x)-2)\frac{\left\langle D^{2}uDu,Du\right\rangle }{\left|Du\right|^{2}}
\]
is the \textit{normalized $p(x)$-Laplacian}, $p:B_{1}\rightarrow\mathbb{R}$
is Lipschitz continuous, $1<p_{\min}:=\inf_{B_{1}}p\leq\sup_{B_{1}}p=:p_{\max}$
and $f\in C(B_{1})$ is bounded. Our main result is that viscosity
solutions to (\ref{eq:normalized p(x)}) are locally $C^{1,\alpha}$-regular.

Normalized equations have attracted a significant amount of interest
during the last 15 years. Their study is partially motivated by their
connection to game theory. Roughly speaking, the value function of
certain stochastic tug-of-war games converges uniformly up to a subsequence
to a viscosity solution of a normalized equation as the step-size
of the game approaches zero \cite{peresShefield08,manfrediParviainenRossi10,manfrediParviainenRossi12,banerjeeGarofalo15,blancRossi19}.
In particular, a game with space-dependent probabilities leads to
the normalized $p(x)$-Laplace equation \cite{arroyoHeinoParviainen17}
and games with running pay-offs lead to inhomogeneous equations \cite{ruosteenoja16}.
In addition to game theory, normalized equations have been studied
for example in the context of image processing \cite{does11,elmoatazToutainTenbrinck15}.

The variable $p(x)$ in (\ref{eq:normalized p(x)}) has an effect
that may not be immediately obvious: If we formally multiply the equation
by $\left|Du\right|^{p(x)-2}$ and rewrite it in a divergence form,
then a logarithm term appears and we arrive at the expression
\begin{equation}
-\div(\left|Du\right|^{p(x)-2}Du)+\left|Du\right|^{p(x)-2}\log(\left|Du\right|)Du\cdot Dp=\left|Du\right|^{p(x)-2}f(x).\label{eq:strong p(x)}
\end{equation}
For $f\equiv0$, this is the so called \textit{strong $p(x)$-Laplace
equation} introduced by Adamowicz and Hästö \cite{adamowiczHasto10,adamowiczHasto11}
in connection with mappings of finite distortion. In the homogeneous
case viscosity solutions to (\ref{eq:normalized p(x)}) actually coincide
with weak solutions of (\ref{eq:strong p(x)}) \cite{siltakoski18},
yielding the $C^{1,\alpha}$-regularity of viscosity solutions as
a consequence of a result by Zhang and Zhou \cite{strongpx_regularity}. 

In the present paper our objective is to prove $C^{1,\alpha}$-regularity
of solutions to (\ref{eq:normalized p(x)}) directly using viscosity
methods. The Hölder regularity of solutions already follows from existing
general results, see \cite{krylovSafonov79,krylovSafonov80,caffarelli89,caffarelliCabre}.
More recently, Imbert and Silvestre \cite{imbertSilvestre12} proved
the gradient Hölder regularity of solutions to the elliptic equation
\[
\left|Du\right|^{\gamma}F(D^{2}u)=f,
\]
where $\gamma>0$ and Imbert, Jin and Silvestre \cite{jinsilvestre17,imbertJinSilvestre16}
obtained a similar result for the parabolic equation
\[
\partial_{t}u=\left|Du\right|^{\gamma}\Delta_{p}^{N}u,
\]
where $p>1$, $\gamma>-1$. Furthermore, Attouchi and Parviainen \cite{attouchiParv}
proved the $C^{1,\alpha}$-regularity of solutions to the inhomogeneous
equation $\partial_{t}u-\Delta_{p}^{N}u=f(x,t)$. Our proof of Hölder
gradient regularity for solutions of (\ref{eq:normalized p(x)}) is
in particular inspired by the papers \cite{jinsilvestre17} and \cite{attouchiParv}.

We point out that recently Fang and Zhang \cite{fangZhang21b} proved
the $C^{1,\alpha}$-regularity of solutions to the parabolic normalized
$p(x,t)$-Laplace equation
\begin{equation}
\partial_{t}u=\Delta_{p(x,t)}^{N}u,\label{eq:parabolic normalized p(x)}
\end{equation}
where $p\in C_{\text{loc}}^{1}$. The equation (\ref{eq:parabolic normalized p(x)})
naturally includes (\ref{eq:normalized p(x)}) if $f\equiv0$. However,
in this article we consider the inhomogeneous case and only suppose
that $p$ is Lipschitz continuous. More precisely, we have the following
theorem.
\begin{thm}
\label{thm:main-1} Suppose that $p$ is Lipschitz continuous in $B_{1}$,
$p_{\min}>1$ and $f\in C(B_{1})$ is bounded. Let $u$ be a viscosity
solution to 
\[
-\Delta_{p(x)}^{N}u=f(x)\quad\text{in }B_{1}.
\]
Then there is $\alpha(N,p_{\min},p_{\max},p_{L})\in(0,1)$ such that
\[
\left\Vert u\right\Vert _{C^{1,\alpha}(B_{1/2})}\leq C(N,p_{\min},p_{\max},p_{L},\left\Vert f\right\Vert _{L^{\infty}(B_{1})},\left\Vert u\right\Vert _{L^{\infty}(B_{1})}),
\]
where $p_{L}$ is the Lipschitz constant of $p$.
\end{thm}

The proof of Theorem \ref{thm:main-1} is based on suitable uniform
$C^{1,\alpha}$-regularity estimates for solutions of the regularized
equation 
\begin{equation}
-\Delta v-(p_{\varepsilon}(x)-2)\frac{\left\langle D^{2}vDv,Dv\right\rangle }{\left|Dv\right|^{2}+\varepsilon^{2}}=g(x),\label{eq:intro regularized}
\end{equation}
where it is assumed that $g$ is continuous and $p_{\varepsilon}$
is smooth. In particular, we show estimates that are independent of
$\varepsilon$ and only depend on $N$, $\sup p$, $\inf p$, $\left\Vert Dp_{\varepsilon}\right\Vert _{L^{\infty}}$
and $\left\Vert g\right\Vert _{L^{\infty}}$. To prove such estimates,
we first derive estimates for the perturbed homogeneous equation
\begin{equation}
-\Delta v-(p_{\varepsilon}(x)-2)\frac{\left\langle D^{2}v(Dv+q),Dv+q\right\rangle }{\left|Dv\right|^{2}+\varepsilon^{2}}=0,\label{eq:intro homogeneous}
\end{equation}
where $q\in\mathbb{R}^{N}$. Roughly speaking, $C^{1,\alpha}$-estimates
for solutions of (\ref{eq:intro homogeneous}) are based on ``improvement
of oscillation'' which is obtained by differentiating the equation
and observing that a function depending on the gradient of the solution
is a supersolution to a linear equation. The uniform $C^{1,\alpha}$-estimates
for solutions of (\ref{eq:intro homogeneous}) then yield uniform
estimates for the inhomogeneous equation (\ref{eq:intro regularized})
by an adaption of the arguments in \cite{imbertSilvestre12,attouchiParv}.

With the \textit{a priori} regularity estimates at hand, the plan
is to let $\varepsilon\rightarrow0$ and show that the estimates pass
on to solutions of (\ref{eq:normalized p(x)}). A problem is caused
by the fact that, to the best of our knowledge, uniqueness of solutions
to (\ref{eq:normalized p(x)}) is an open problem for variable $p(x)$
and even for constant $p$ if $f$ is allowed to change signs. To
deal with this, we fix a solution $u_{0}\in C(\overline{B}_{1})$
to (\ref{eq:normalized p(x)}) and consider the Dirichlet problem
\begin{equation}
-\Delta_{p(x)}^{N}u=f(x)-u_{0}(x)-u\quad\text{in }B_{1}\label{eq:intro dirichlet}
\end{equation}
with boundary data $u=u_{0}$ on $\partial B_{1}$. For this equation
the comparison principle holds and thus $u_{0}$ is the unique solution.
We then consider the approximate problem
\begin{equation}
-\Delta u_{\varepsilon}-(p_{\varepsilon}(x)-2)\frac{\left\langle D^{2}u_{\varepsilon}Du_{\varepsilon},Du_{\varepsilon}\right\rangle }{\left|Du_{\varepsilon}\right|^{2}+\varepsilon^{2}}=f_{\varepsilon}(x)-u_{0,\varepsilon}(x)-u_{\varepsilon}\label{eq:intro regularized 2}
\end{equation}
with boundary data $u_{\varepsilon}=u_{0}$ on $\partial B_{1}$ and
where $p_{\varepsilon},f_{\varepsilon},u_{0,\varepsilon}\in C^{\infty}(B_{1})$
are such that $p\rightarrow p_{\varepsilon}$, $f_{\varepsilon}\rightarrow f$
and $u_{0,\varepsilon}\rightarrow u_{0}$ uniformly in $B_{1}$ and
$\left\Vert Dp_{\varepsilon}\right\Vert _{L^{\infty}(B_{1})}\leq\left\Vert Dp\right\Vert _{L^{\infty}(B_{1})}$.
As the equation (\ref{eq:intro regularized 2}) is uniformly elliptic
quasilinear equation with smooth coefficients, the solution $u_{\varepsilon}$
exists in the classical sense by standard theory. Since $u_{\varepsilon}$
also solves (\ref{eq:intro regularized}) with $g(x)=f_{\varepsilon}(x)-u_{0,\varepsilon}(x)-u_{\varepsilon}(x)$,
it satisfies the uniform $C^{1,\alpha}$-regularity estimate. We then
let $\varepsilon\rightarrow0$ and use stability and comparison principles
to show that $u_{0}$ inherits the regularity estimate.

For other related results, see for example the works of Attouchi,
Parviainen and Ruosteenoja \cite{OptimalC1} on the normalized $p$-Poisson
problem $-\Delta_{p}^{N}u=f$, Attouchi and Ruosteenoja \cite{attouchiRuosteenoja18,attouchiRuosteenoja20,attouchi20}
on the equation $-\left|Du\right|^{\gamma}\Delta_{p}^{N}u=f$ and
its parabolic version, De Filippis \cite{deflippis21} on the double
phase problem $(\left|Du\right|^{q}+a(x)\left|Du\right|^{s})F(D^{2}u)=f(x)$
and Fang and Zhang \cite{fangZhang21} on the parabolic double phase
problem $\partial_{t}u=(\left|Du\right|^{q}+a(x,t)\left|Du\right|^{s})\Delta_{p}^{N}u$.
We also mention the paper by Bronzi, Pimentel, Rampasso and Teixeira
\cite{bronziPimentelRampassoTeixeira} where they consider fully nonlinear
variable exponent equations of the type $\left|Du\right|^{\theta(x)}F(D^{2}u)=0$. 

The paper is organized as follows: Section 2 is dedicated to preliminaries,
Sections 3 and 4 contain $C^{1,\alpha}$-regularity estimates for
equations (\ref{eq:intro homogeneous}) and (\ref{eq:intro regularized 2}),
and Section 5 contains the proof of Theorem (\ref{thm:main-1}). Finally,
the Appendix contains an uniform Lipschitz estimate for the equations
studied in this paper and a comparison principle for equation (\ref{eq:intro dirichlet}).

\section{Preliminaries}

\subsection{Notation}

We denote by $B_{R}\subset\mathbb{R}^{N}$ an open ball of radius
$R>0$ that is centered at the origin in the $N$-dimensional Euclidean
space, $N\geq1$. The set of symmetric $N\times N$ matrices is denoted
by $S^{N}$. For $X,Y\in S^{N}$, we write $X\leq Y$ if $X-Y$ is
negative semidefinite. We also denote the smallest eigenvalue of $X$
by $\lambda_{\min}(X)$ and the largest by $\lambda_{\max}(X)$ and
set
\[
\left\Vert X\right\Vert :=\sup_{\xi\in B_{1}}\left|X\xi\right|=\sup\left\{ \left|\lambda\right|:\lambda\text{ is an eigenvalue of }X\right\} .
\]
We use the notation $C(a_{1},\ldots,a_{k})$ to denote a constant
$C$ that may change from line to line but depends only on $a_{1},\ldots,a_{k}$.
For convenience we often use $C(\hat{p})$ to mean that the constant
may depend on $p_{\min}$, $p_{\max}$ and the Lipschitz constant
$p_{L}$ of $p$.

For $\alpha\in(0,1)$, we denote by $C^{\alpha}(B_{R})$ the set of
all functions $u:B_{R}\rightarrow\mathbb{R}$ with finite Hölder norm
\[
\left\Vert u\right\Vert _{C^{\alpha}(B_{R})}:=\left\Vert u\right\Vert _{L^{\infty}(B_{R})}+\left[u\right]_{C^{\alpha}(B_{R})},\text{\ensuremath{\quad}where }\left[u\right]_{C^{\alpha}(B_{R})}:=\sup_{x,y\in B_{R}}\frac{\left|u(x)-u(y)\right|}{\left|x-y\right|^{\alpha}}.
\]
Similarly, we denote by $C^{1,\alpha}(B_{R})$ the set of all functions
for which the norm
\[
\left\Vert u\right\Vert _{C^{1,\alpha}(B_{R})}:=\left\Vert u\right\Vert _{C^{\alpha}(B_{R})}+\left\Vert Du\right\Vert _{C^{\alpha}(B_{R})}
\]
 is finite.

\subsection{Viscosity solutions}

Viscosity solutions are defined using smooth test functions that touch
the solution from above or below. If $u,\varphi:\mathbb{R}^{N}\rightarrow\mathbb{R}$
and $x\in\mathbb{R}^{N}$ are such that $\varphi(x)=u(x)$ and $\varphi(y)<u(y)$
for $y\not=x_{0}$, then we say that \textit{$\varphi$ touches $u$
from below at $x_{0}$}.
\begin{defn}
\label{def:viscosity solutions} Let $\Omega\subset\mathbb{R}^{N}$
be a bounded domain. Suppose that $f:\Omega\times\mathbb{R}\rightarrow\mathbb{R}$
is continuous. A lower semicontinuous function $u:\Omega\rightarrow\mathbb{R}$
is a \textit{viscosity supersolution} to 
\[
-\Delta_{p(x)}^{N}u\geq f(x,u)\quad\text{in }\Omega
\]
if the following holds: Whenever $\varphi\in C^{2}(\Omega)$ touches
$u$ from below at $x\in\Omega$ and $D\varphi(x)\not=0$, we have
\[
-\Delta\varphi(x)-(p(x)-2)\frac{\left\langle D^{2}\varphi(x)D\varphi(x),D\varphi(x)\right\rangle }{\left|D\varphi(x)\right|^{2}}\geq f(x,u(x))
\]
and if $D\varphi(x)=0$, then
\[
-\Delta\varphi(x)-(p(x)-2)\left\langle D^{2}\varphi(x)\eta,\eta\right\rangle \ge f(x,u(x))\quad\text{for some }\eta\in\overline{B}_{1}.
\]
Analogously, a lower semicontinuous function $u:\Omega\rightarrow\mathbb{R}$
is a viscosity subsolution if the above inequalities hold reversed
whenever $\varphi$ touches $u$ from above. Finally, we say that
$u$ is a \textit{viscosity solution} if it is both viscosity sub-
and supersolution.
\end{defn}

\begin{rem*}
The special treatment of the vanishing gradient in Definition \ref{def:viscosity solutions}
is needed because of the singularity of the equation. Definition \ref{def:viscosity solutions}
is essentially a relaxed version of the standard definition in \cite{userguide}
which is based on the so called semicontinuous envelopes. In the standard
definition one would require that if $\varphi$ touches a viscosity
supersolution $u$ from below at $x$, then
\[
\begin{cases}
-\Delta_{p(x)}^{N}\varphi(x)\geq f(x,u(x)) & \text{if }D\varphi(x)\not=0,\\
-\Delta\varphi(x)-(p(x)-2)\lambda_{\min}(D^{2}\varphi(x))\geq f(x,u(x)) & \text{if }D\varphi(x)=0\text{ and }p(x)\geq2,\\
-\Delta\varphi(x)-(p(x)-2)\lambda_{\max}(D^{2}\varphi(x))\geq f(x,u(x)) & \text{if }D\varphi(x)=0\text{ and }p(x)<2.
\end{cases}
\]
Clearly, if $u$ is a viscosity supersolution in this sense, then
it is also a viscosity supersolution in the sense of Definition \ref{def:viscosity solutions}.
\end{rem*}

\section{Hölder gradient estimates for the regularized homogeneous equation\label{sec:regularized homogeneous}}

In this section we prove $C^{1,\alpha}$-regularity estimates for
solutions to the equation
\begin{equation}
-\Delta u-(p(x)-2)\frac{\left\langle D^{2}u(Du+q),Du+q\right\rangle }{\left|Du+q\right|^{2}+\varepsilon^{2}}=0\quad\text{in }B_{1},\label{eq:regularized homogeneous}
\end{equation}
where $p:B_{1}\rightarrow B_{1}$ is Lipschitz, $p_{\min}>1$, $\varepsilon>0$
and $q\in\mathbb{R}^{N}$. Our objective is to obtain estimates that
are independent of $q$ and $\varepsilon$. Observe that (\ref{eq:regularized homogeneous})
is a uniformly elliptic quasilinear equation with smooth coefficients.
Viscosity solutions to (\ref{eq:regularized homogeneous}) can be
defined in the standard way and they are smooth if $p$ is smooth.
\begin{prop}
\label{prop:c infty} Suppose that $p$ is smooth. Let $u$ be a viscosity
solution to (\ref{eq:regularized homogeneous}) in $B_{1}$. Then
$u\in C^{\infty}(B_{1})$.
\end{prop}

It follows from classical theory that the corresponding Dirichlet
problem admits a smooth solution (see \cite[Theorems 15.18 and 13.6]{gilbargTrudinger01}
and the Schauder estimates \cite[Theorem 6.17]{gilbargTrudinger01}).
The viscosity solution $u$ coincides with the smooth solution by
a comparison principle \cite[Theorem 3]{kawohlNikolai98}.

\subsection{Improvement of oscillation}

Our regularity estimates for solutions of (\ref{eq:regularized homogeneous})
are based on improvement of oscillation. We first prove such a result
for the linear equation
\begin{equation}
-\tr(G(x)D^{2}u)=f\quad\text{in }B_{1},\label{eq:linear equation}
\end{equation}
where $f\in C^{1}(B_{1})$ is bounded, $G(x)\in S^{N}$ and there
are constants $0<\lambda<\varLambda<\infty$ such that the eigenvalues
of $G(x)$ are in $[\lambda,\varLambda]$ for all $x\in B_{1}$. The
result is based on the following rescaled version of the weak Harnack
inequality found in \cite[Theorem 4.8]{caffarelliCabre}. Such Harnack
estimates for non-divergence form equations go back to at least Krylov
and Safonov \cite{krylovSafonov79,krylovSafonov80}.
\begin{lem}[Weak Harnack inequality]
\label{lem:weak harnack} Let $u\geq0$ be a continuous viscosity
supersolution to (\ref{eq:linear equation}) in $B_{1}$. Then there
are positive constants $C(\lambda,\varLambda,N)$ and $q(\lambda,\varLambda,N)$
such that for any $\tau<\frac{1}{4\sqrt{N}}$ we have
\begin{equation}
\tau^{-\frac{N}{q}}\left(\int_{B_{\tau}}\left|u\right|^{q}\d x\right)^{1/q}\leq C\left(\inf_{B_{2\tau}}u+\tau\left(\int_{B_{4\sqrt{N}\tau}}\left|f\right|^{N}\d x\right)^{1/N}\right).\label{eq:weak harnack}
\end{equation}
\end{lem}

\begin{proof}
Suppose that $\tau<\frac{1}{4\sqrt{N}}$ and set $S:=8\tau$. Define
the function $v:B_{\sqrt{N}/2}\rightarrow\mathbb{R}$ by
\begin{align*}
v(x) & :=u(Sx)
\end{align*}
and set
\[
\tilde{G}(x):=G(Sx)\quad\text{and}\quad\tilde{f}(x):=S^{2}f(Sx).
\]
Then, if $\varphi\in C^{2}$ touches $v$ from below at $x\in B_{\sqrt{N}/2}$,
the function $\phi(x):=\varphi(x/S)$ touches $u$ from below at $Sx$.
Therefore
\[
-\tr(G(Sx)D^{2}\phi(Sx))\geq f(Sx).
\]
Since $D^{2}\phi(Sx)=S^{-2}D^{2}\varphi(x)$, this implies that
\[
-\tr(G(Sx)D^{2}\varphi(x))\geq S^{2}f(Sx).
\]
Thus $v$ is a viscosity supersolution to
\[
-\tr(\tilde{G}(x)D^{2}v)\geq\tilde{f}(x)\quad\text{in }B_{\sqrt{N}/2}.
\]
We denote by $Q_{R}$ a cube with side-length $R/2$. Since $Q_{1}\subset B_{\sqrt{N}/2}$,
it follows from \cite[Theorem 4.8]{caffarelliCabre} that there are
$q(\lambda,\varLambda,N)$ and $C(\lambda,\varLambda,N)$ such that
\begin{align*}
\left(\int_{B_{1/8}}\left|v\right|^{q}\d x\right)^{1/q}\leq\left(\int_{Q_{1/4}}\left|v\right|^{q}\d x\right)^{1/q} & \leq C\left(\inf_{Q_{1/2}}v+\left(\int_{Q_{1}}|\tilde{f}|^{N}\d x\right)^{1/N}\right)\\
 & \leq C\left(\inf_{B_{1/4}}v+\left(\int_{B_{\sqrt{N}/2}}|\tilde{f}|^{N}\d x\right)^{1/N}\right).
\end{align*}
By the change of variables formula we have
\begin{align*}
\int_{B_{1/8}}\left|v\right|^{q}\d x= & \int_{B_{1/8}}\left|u(Sx)\right|^{q}\d x=S^{-N}\int_{B_{S/8}}\left|u(x)\right|^{q}\d x
\end{align*}
and
\[
\int_{B_{\sqrt{N}/2}}|\tilde{f}|^{N}\d x=S^{2N}\int_{B_{\sqrt{N}/2}}\left|f(Sx)\right|^{N}\d x=S^{N}\int_{B_{S\sqrt{N}/2}}\left|f(x)\right|^{N}\d x.
\]
Recalling that $S=8\tau$, we get
\[
8^{-\frac{N}{q}}\tau^{-\frac{N}{q}}\left(\int_{B_{\tau}}\left|u(x)\right|^{q}\d x\right)^{1/q}\leq C\left(\inf_{B_{2\tau}}u+8\tau\left(\int_{B_{S\sqrt{N}/2}}\left|f(x)\right|^{N}\d x\right)^{1/N}\right).
\]
Absorbing $8^{\frac{N}{q}}$ into the constant, we obtain the claim.
\end{proof}
\begin{lem}[Improvement of oscillation for the linear equation]
\label{lem:imposc linear} Let $u\geq0$ be a continuous viscosity
supersolution to $\eqref{eq:linear equation}$ in $B_{1}$ and $\mu,l>0$.
Then there are positive constants $\tau(\lambda,\varLambda,N,\mu,l,\left\Vert f\right\Vert _{L^{\infty}(B_{1})})$
and $\theta(\lambda,\varLambda,N,\mu,l)$ such that if
\begin{equation}
\left|\left\{ x\in B_{\tau}:u\ge l\right\} \right|>\mu\left|B_{\tau}\right|,\label{eq:improvement of oscillation 1}
\end{equation}
then we have
\[
u\geq\theta\quad\text{in }B_{\tau}.
\]
\end{lem}

\begin{proof}
By the weak Harnack inequality (Lemma \ref{lem:weak harnack}) there
exist constants $C_{1}(\lambda,\varLambda,N)$ and $q(\lambda,\varLambda,N)$
such that for any $\tau<1/(4\sqrt{N}),$ we have
\begin{equation}
\inf_{B_{2\tau}}u\geq C_{1}\tau^{\frac{-N}{q}}\left(\int_{B_{\tau}}\left|u\right|^{q}\d x\right)^{1/q}-\tau\left(\int_{B_{4\sqrt{N}\tau}}\left|f\right|^{N}\d x\right)^{1/N}.\label{eq:improvement of oscillation 2}
\end{equation}
In particular, this holds for 
\[
\tau:=\min\left(\frac{1}{4\sqrt{N}},\sqrt{\frac{C_{1}\left|B_{1}\right|^{\frac{1}{q}-\frac{1}{N}}\mu^{\frac{1}{q}}l}{2\cdot4\sqrt{N}(\left\Vert f\right\Vert _{L^{\infty}(B_{1})}+1)}}\right).
\]
We continue the estimate (\ref{eq:improvement of oscillation 2})
using the assumption (\ref{eq:improvement of oscillation 1}) and
obtain
\begin{align*}
\inf_{B_{\tau}}u\geq\inf_{B_{2\tau}}u & \geq C_{1}\tau^{-\frac{N}{q}}\left(\left|\left\{ x\in B_{\tau}:u\geq l\right\} \right|l^{q}\right)^{1/q}-\tau\left(\int_{B_{4\sqrt{N}\tau}}\left|f\right|^{N}\d x\right)^{1/N}\\
 & \geq C_{1}\tau^{-\frac{N}{q}}\mu^{\frac{1}{q}}\left|B_{\tau}\right|^{\frac{1}{q}}l-\tau\left|B_{4\sqrt{N}\tau}\right|^{\frac{1}{N}}\left\Vert f\right\Vert _{L^{\infty}(B_{1})}\\
 & =C_{1}\left|B_{1}\right|^{\frac{1}{q}}\mu^{\frac{1}{q}}l\tau^{-\frac{N}{q}}\tau^{\frac{N}{q}}-4\sqrt{N}\left|B_{1}\right|^{\frac{1}{N}}\left\Vert f\right\Vert _{L^{\infty}(B_{1})}\tau^{2}\\
 & =C_{1}\left|B_{1}\right|^{\frac{1}{q}}\mu^{\frac{1}{q}}l-4\sqrt{N}\left|B_{1}\right|^{\frac{1}{N}}\left\Vert f\right\Vert _{L^{\infty}(B_{1})}\tau^{2}.\\
 & \geq\frac{1}{2}C_{1}\left|B_{1}\right|^{\frac{1}{q}}\mu^{\frac{1}{q}}l,\\
 & =:\theta,
\end{align*}
where the last inequality follows from the choice of $\tau$. 
\end{proof}
We are now ready to prove an improvement of oscillation for the gradient
of a solution to (\ref{eq:regularized homogeneous}). We first consider
the following lemma, where the improvement is considered towards a
fixed direction. We initially also restrict the range of $\left|q\right|$.

The idea is to differentiate the equation and observe that a suitable
function of $Du$ is a supersolution to the linear equation (\ref{eq:linear equation}).
Lemma \ref{lem:imposc linear} is then applied to obtain information
about $Du$.
\begin{lem}[Improvement of oscillation to direction]
\label{lem:imposc dir} Suppose that $p$ is smooth. Let $u$ be
a smooth solution to (\ref{eq:regularized homogeneous}) in $B_{1}$
with $\left|Du\right|\leq1$ and either $q=0$ or $\left|q\right|>2$.
Then for every $0<l<1$ and $\mu>0$ there exist positive constants
$\tau(N,\hat{p},l,\mu)<1$ and $\gamma(N,\hat{p},l,\mu)<1$ such that
\[
\left|\left\{ x\in B_{\tau}:Du\cdot d\leq l\right\} \right|>\mu\left|B_{\tau}\right|\quad\text{implies\ensuremath{\quad}}Du\cdot d\leq\gamma\text{ in }B_{\tau}
\]
whenever $d\in\partial B_{1}$.
\end{lem}

\begin{proof}
To simplify notation, we set
\begin{align*}
A_{ij}(x,\eta) & :=\delta_{ij}+(p(x)-2)\frac{(\eta_{i}+q_{i})(\eta_{j}+q_{j})}{\left|\eta+q\right|^{2}+\varepsilon^{2}}.
\end{align*}
We also denote the functions $\mathcal{A}_{ij}:x\mapsto A_{ij}(x,Du(x))$,
$\mathcal{A}_{ij,x_{k}}:x\mapsto(\partial_{x_{k}}A_{ij})(x,Du(x))$
and $\mathcal{A}_{ij,\eta_{k}}:x\mapsto(\partial_{\eta_{k}}A_{ij})(x,Du(x))$.
Then, since $u$ is a smooth solution to (\ref{eq:regularized homogeneous})
in $B_{1}$, we have in Einstein's summation convention
\[
-\mathcal{A}_{ij}u_{ij}=0\quad\text{pointwise in }B_{1}.
\]
Differentiating this yields
\begin{align}
0=(\mathcal{A}_{ij}u_{ij})_{k} & =\mathcal{A}_{ij}u_{ijk}+(\mathcal{A}_{ij})_{k}u_{ij}\nonumber \\
 & =\mathcal{A}_{ij}u_{ijk}+\mathcal{A}_{ij,\eta_{m}}u_{ij}u_{km}+\mathcal{A}_{ij,x_{k}}u_{ij}\quad\text{for all }k=1,\ldots N.\label{eq:imposc dir 2}
\end{align}
Multiplying these identities by $d_{k}$ and summing over $k$, we
obtain
\begin{align}
0 & =\mathcal{A}_{ij}u_{ijk}d_{k}+\mathcal{A}_{ij,\eta_{m}}u_{ij}u_{km}d_{k}+\mathcal{A}_{ij,x_{k}}u_{ij}d_{k}\nonumber \\
 & =\mathcal{A}_{ij}(Du\cdot d-l)_{ij}+\mathcal{A}_{ij,\eta_{m}}u_{ij}(Du\cdot d-l)_{m}+\mathcal{A}_{ij,x_{k}}u_{ij}d_{k}.\label{eq:imposc dir 3}
\end{align}
Moreover, multiplying (\ref{eq:imposc dir 2}) by $2u_{k}$ and summing
over $k$, we obtain
\begin{align}
0 & =2\mathcal{A}_{ij}u_{ijk}u_{k}+2\mathcal{A}_{ij,\eta_{m}}u_{ij}u_{km}u_{k}+2\mathcal{A}_{ij,x_{k}}u_{ij}u_{k}\nonumber \\
 & =\mathcal{A}_{ij}(2u_{ijk}u_{k}+2u_{kj}u_{ki})-2\mathcal{A}_{ij}u_{kj}u_{ki}+2\mathcal{A}_{ij,\eta_{m}}u_{ij}u_{km}u_{k}+2\mathcal{A}_{ij,x_{k}}u_{ij}u_{k}\nonumber \\
 & =\mathcal{A}_{ij}(u_{k}^{2})_{ij}-2\mathcal{A}_{ij}u_{kj}u_{ki}+\mathcal{A}_{ij,\eta_{m}}u_{ij}(u_{k}^{2})_{m}+2\mathcal{A}_{ij,x_{k}}u_{ij}u_{k}\nonumber \\
 & =\mathcal{A}_{ij}(\left|Du\right|^{2})_{ij}+\mathcal{A}_{ij,\eta_{m}}u_{ij}(\left|Du\right|^{2})_{m}+2\mathcal{A}_{ij,x_{k}}u_{ij}u_{k}-2\mathcal{A}_{ij}u_{kj}u_{ki}.\label{eq:imposc dir 4}
\end{align}
We will now split the proof into the cases $q=0$ or $\left|q\right|>2$,
and proceed in two steps: First we check that a suitable function
of $Du$ is a supersolution to the linear equation (\ref{lem:imposc linear})
and then apply Lemma \ref{lem:imposc linear} to obtain the claim.

\textbf{Case $q=0$, Step 1: }We denote $\Omega_{+}:=\left\{ x\in B_{1}:h(x)>0\right\} $,
where
\[
h:=(Du\cdot d-l+\frac{l}{2}\left|Du\right|^{2})^{+}.
\]
If $\left|Du\right|\leq l/2$, we have
\[
Du\cdot d-l+\frac{l}{2}\left|Du\right|^{2}\leq-\frac{l}{2}+\frac{l^{3}}{8}<0.
\]
This implies that $\left|Du\right|>l/2$ in $\Omega_{+}$. Therefore,
since $q=0$, we have in $\Omega_{+}$
\begin{align}
\left|\mathcal{A}_{ij,\eta_{m}}\right| & =\left|p(x)-2\right|\left|\frac{\delta_{im}(u_{j}+q_{j})+\delta_{jm}(u_{i}+q_{i})}{\left|Du+q\right|^{2}+\varepsilon^{2}}-\frac{2(u_{m}+q_{m})(u_{i}+q_{i})(u_{j}+q_{j})}{(\left|Du+q\right|^{2}+\varepsilon^{2})^{2}}\right|\nonumber \\
 & \leq8l^{-1}\left\Vert p-2\right\Vert _{L^{\infty}(B_{1})},\label{eq:imposc dir 5}\\
\left|\mathcal{A}_{ij,x_{k}}\right| & =\left|Dp(x)\right|\left|\frac{(\eta_{i}+q_{i})(\eta_{j}+q_{j})}{\left|\eta+q\right|^{2}+\varepsilon^{2}}\right|\leq p_{L}.\label{eq:imposc dir 6}
\end{align}
Summing up the equations (\ref{eq:imposc dir 3}) and (\ref{eq:imposc dir 4})
multiplied by $2^{-1}l$, we obtain in $\Omega_{+}$
\begin{align*}
0=\  & \mathcal{A}_{ij}(Du\cdot d-l)_{ij}+\mathcal{A}_{ij,\eta_{m}}u_{ij}(Du\cdot d-l)_{m}+\mathcal{A}_{ij,x_{k}}u_{ij}d_{k}\\
 & +2^{-1}l\big(\mathcal{A}_{ij}(\left|Du\right|^{2})_{ij}+\mathcal{A}_{ij,\eta_{m}}u_{ij}(\left|Du\right|^{2})_{m}+2\mathcal{A}_{ij,x_{k}}u_{ij}u_{k}-2\mathcal{A}_{ij}u_{kj}u_{ki}\big)\\
=\  & \mathcal{A}_{ij}h_{ij}+\mathcal{A}_{ij,\eta_{m}}u_{ij}h_{m}+\mathcal{A}_{ij,x_{k}}u_{ij}d_{k}+l\mathcal{A}_{ij,x_{k}}u_{ij}u_{k}-l\mathcal{A}_{ij}u_{kj}u_{ki}\\
\leq\  & \mathcal{A}_{ij}h_{ij}+\left|\mathcal{A}_{ij,\eta_{m}}u_{ij}\right|\left|h_{m}\right|+\left|\mathcal{A}_{ij,x_{k}}u_{ij}\right|\left|d_{k}+lu_{k}\right|-l\mathcal{A}_{ij}u_{kj}u_{ki}.
\end{align*}
Since $\left|Du\right|\leq1$, we have $\left|d_{k}+lu_{k}\right|^{2}\leq4$
and by uniform ellipticity $\mathcal{A}_{ij}u_{kj}u_{ki}\geq\min(p_{\min}-1,1)\left|u_{ij}\right|^{2}$.
Therefore, by applying Young's inequality with $\epsilon>0$, we obtain
from the above display
\begin{align*}
0 & \leq\mathcal{A}_{ij}h_{ij}+N^{2}\epsilon^{-1}(\left|h_{m}\right|^{2}+\left|d_{k}+lu_{k}\right|^{2})+\epsilon(\left|\mathcal{A}_{ij,\eta_{m}}\right|^{2}+\left|\mathcal{A}_{ij,x_{k}}\right|^{2})\left|u_{ij}\right|^{2}-l\mathcal{A}_{ij}u_{kj}u_{ki}\\
 & \leq\mathcal{A}_{ij}h_{ij}+N^{2}\epsilon^{-1}(\left|Dh\right|^{2}+4)+\epsilon C(N,\hat{p})(l^{-2}+1)\left|u_{ij}\right|^{2}-l\min(\pmin-1,1)\left|u_{ij}\right|^{2},
\end{align*}
where in the second estimate we used (\ref{eq:imposc dir 5}) and
(\ref{eq:imposc dir 6}). By taking $\epsilon$ small enough, we obtain
\begin{equation}
0\leq\mathcal{A}_{ij}h_{ij}+C_{0}(N,\hat{p})\frac{\left|Dh\right|^{2}+1}{l^{3}}\quad\text{in }\Omega_{+},\label{eq:imposc dir 7}
\end{equation}
Next we define 
\begin{equation}
\overline{h}:=\frac{1}{\nu}(1-e^{\nu(h-H)}),\quad\text{where}\quad H:=1-\frac{l}{2}\quad\text{and}\quad\nu:=\frac{C_{0}}{l^{3}\min(\pmin-1,1)}.\label{eq:imposc dir 10}
\end{equation}
Then by (\ref{eq:imposc dir 7}) and uniform ellipticity we have in
$\Omega_{+}$
\begin{align*}
-\mathcal{A}_{ij}\overline{h}{}_{ij} & =\mathcal{A}_{ij}(h_{ij}e^{\nu(h-H)}+\nu h_{i}h_{j}e^{\nu(h-H)})\\
 & \geq e^{\nu(h-H)}(-C_{0}\frac{\left|Dh\right|^{2}}{l^{3}}-\frac{C_{0}}{l^{3}}+\nu\min(\pmin-1,1)\left|Dh\right|^{2})\\
 & \geq-\frac{C_{0}}{l^{3}}.
\end{align*}
Since the minimum of two viscosity supersolutions is still a viscosity
supersolution, it follows from the above estimate that $\overline{h}$
is a non-negative viscosity supersolution to
\begin{equation}
-\mathcal{A}_{ij}\overline{h}_{ij}\ge\frac{-C_{0}}{l^{3}}\quad\text{in }B_{1}.\label{eq:imposc dir 8}
\end{equation}

\textbf{Case $q=0$, Step 2: }We set $l_{0}:=\frac{1}{\nu}(1-e^{\nu(l-1)})$.
Then, since $\overline{h}$ solves (\ref{eq:imposc dir 8}), by Lemma
\ref{lem:imposc linear} there are positive constants $\tau(N,p,l,\mu)$
and $\theta(N,p,l,\mu)$ such that
\[
\left|\left\{ x\in B_{\tau}:\overline{h}\geq l_{0}\right\} \right|>\mu\left|B_{\tau}\right|\quad\text{implies}\quad\overline{h}\geq\theta\quad\text{in }B_{\tau}.
\]
If $Du\cdot d\leq l$, we have $\overline{h}\geq l_{0}$ and therefore
\[
\left|\left\{ x\in B_{\tau}:\overline{h}\geq l_{0}\right\} \right|\geq\left|\left\{ x\in B_{\tau}:Du\cdot d\leq l\right\} \right|>\mu\left|B_{\tau}\right|,
\]
where the last inequality follows from the assumptions. Consequently,
we obtain
\[
\overline{h}\geq\theta\quad\text{in }B_{\tau}.
\]
Since $h-H\leq0$, by convexity we have $H-h\geq\overline{h}$. This
together with the above estimate yields
\[
1-2^{-1}l-(Du\cdot d-l+2^{-1}l\left|Du\right|^{2})\geq\theta\quad\text{in }B_{\tau}
\]
and so
\[
Du\cdot d+2^{-1}l(Du\cdot d)^{2}\leq Du\cdot d+2^{-1}l\left|Du\right|^{2}\leq1+2^{-1}l-\theta\quad\text{in }B_{\tau}.
\]
Using the quadratic formula, we thus obtain the desired estimate
\[
Du\cdot d\leq\frac{-1+\sqrt{1+2l(1+2^{-1}l-\theta)}}{l}=\frac{-1+\sqrt{(1+l)^{2}-2l\theta}}{l}=:\gamma<1\quad\text{in }B_{\tau}.
\]

\textbf{Case $\left|q\right|>2$: }Computing like in (\ref{eq:imposc dir 5})
and (\ref{eq:imposc dir 6}), we obtain this time in $B_{1}$
\[
\left|\mathcal{A}_{ij,\eta_{m}}\right|\leq4\left\Vert p-2\right\Vert _{L^{\infty}(B_{1})}\quad\text{and}\quad\left|\mathcal{A}_{ij,x_{k}}\right|\leq p_{L}
\]
Moreover, this time we set simply
\[
h:=Du\cdot d-l+2^{-1}l\left|Du\right|^{2}.
\]
Summing up the identities (\ref{eq:imposc dir 3}) and (\ref{eq:imposc dir 4})
and using Young's inequality similarly as in the case $\left|q\right|=0$,
we obtain in $B_{1}$
\begin{align*}
0 & \leq\mathcal{A}_{ij}h_{ij}+N^{2}\epsilon^{-1}(\left|h_{m}\right|^{2}+\left|d_{k}+lu_{k}\right|^{2})+\epsilon(\left|\mathcal{A}_{ij,\eta_{m}}\right|^{2}+\left|\mathcal{A}_{ij,x_{k}}\right|^{2})\left|u_{ij}\right|^{2}-l\mathcal{A}_{ij}u_{kj}u_{ki}\\
 & \leq\mathcal{A}_{ij}h_{ij}+N^{2}\epsilon^{-1}(\left|Dh\right|^{2}+4)+\epsilon C(\hat{p})\left|u_{ij}\right|^{2}-lC(\hat{p})\left|u_{ij}\right|^{2}.
\end{align*}
By taking small enough $\epsilon$, we obtain
\[
0\leq\mathcal{A}_{ij}h_{ij}+C_{0}(N,\hat{p})\frac{\left|Dh\right|^{2}+1}{l}\quad\text{in }B_{1}.
\]
Next we define $\overline{h}$ and $H$ like in (\ref{eq:imposc dir 10}),
but set instead $\nu:=C_{0}/(l\min(p_{\min}-1,1))$. The rest of the
proof then proceeds in the same way as in the case $q=0$.\begin{details}
\[
\overline{h}:=\frac{1}{\nu}(1-e^{\nu(h-H)}),\quad\text{where}\quad H:=1-\frac{l}{2}\quad\text{and}\quad\nu:=\frac{C_{0}}{l\min(\pmin-1,1)}.
\]
Then we have in $B_{1}$
\begin{align}
-A_{ij}\overline{h}_{ij} & =A_{ij}(h_{ij}e^{\nu(h-H)}+\nu h_{i}h_{j}e^{\nu(h-H)})\nonumber \\
 & \geq e^{\nu(h-H)}(-C_{0}\frac{\left|Dh\right|^{2}}{l}-C_{0}+\nu\min(\pmin-1,1)\left|Dh\right|^{2})\nonumber \\
 & \geq-C_{0}.\label{eq:imposc dir 9}
\end{align}
\textbf{(Case $\left|q\right|>2$, Step 2)} We set $l_{0}:=\frac{1}{\nu}(1-e^{\nu(l-1)})$.
Then, since $\overline{h}$ solves (\ref{eq:imposc dir 9}), by Lemma
(\ref{eq:imposc dir 9}) there are positive constants $\tau(p,N,l,\mu)$
and $\theta(p,N,l,\mu)$ such that 
\[
\left|\left\{ x\in B_{\tau}:\overline{h}\geq l_{0}\right\} \right|>\mu\left|B_{\tau}\right|\quad\text{implies\ensuremath{\quad}}\overline{h}\ge\theta\quad\text{in }B_{\tau}.
\]
If $Du\cdot d\leq l$, we have $\overline{h}\geq l_{0}$ and therefore
\[
\left|\left\{ x\in B_{\tau}:h\geq l_{0}\right\} \right|\geq\left|\left\{ x\in B_{\tau}:Du\cdot d\leq l\right\} \right|>\mu\left|B_{\tau}\right|,
\]
where the last estimate follows from the assumptions. Consequently
we obtain
\[
\overline{h}\geq\theta\quad\text{in }B_{\tau}.
\]
By convexity we have $H-h\geq\overline{h}.$ This together with the
above estimate yields
\[
1-2^{-1}l-(Du\cdot d-l+2^{-1}l\left|Du\right|^{2})\geq\theta\quad\text{in }B_{\tau}
\]
and so 
\[
Du\cdot d+2^{-1}l(Du\cdot d)^{2}\leq Du\cdot d+2^{-1}l\left|Du\right|^{2}\leq1+2^{-1}l-\theta\quad\text{in }B_{\tau}.
\]
\end{details}
\end{proof}
Next we inductively apply the previous lemma to prove the improvement
of oscillation.
\begin{thm}[Improvement of oscillation]
\label{thm:imposc} Suppose that $p$ is smooth. Let $u$ be a smooth
solution to (\ref{eq:regularized homogeneous}) in $B_{1}$ with $\left|Du\right|\leq1$
and either $q=0$ or $\left|q\right|>2$. Then for every $0<l<1$
and $\mu>0$ there exist positive constants $\tau(N,\hat{p},l,\mu)<1$
and $\gamma(N,\hat{p},l,\mu)<1$ such that if
\begin{equation}
\left|\left\{ x\in B_{\tau^{i+1}}:Du\cdot d\leq l\gamma^{i}\right\} \right|>\mu\left|B_{\tau^{i+1}}\right|\quad\text{for all }d\in\partial B_{1},\ i=0,\ldots,k,\label{eq:imposc cnd}
\end{equation}
then
\begin{equation}
\left|Du\right|\leq\gamma^{i+1}\quad\text{in }B_{\tau^{i+1}}\quad\text{for all }i=0,\ldots,k.\label{eq:imposc cnd2}
\end{equation}
\end{thm}

\begin{proof}
Let $k\geq0$ be an integer and suppose that (\ref{eq:imposc cnd})
holds. We proceed by induction.\textbf{ }

\textbf{Initial step:} Since (\ref{eq:imposc cnd}) holds for $i=0$,
by Lemma \ref{lem:imposc dir} we have $Du\cdot d\leq\gamma$ in $B_{\tau}$
for all $d\in\partial B_{1}$. This implies (\ref{eq:imposc cnd2})
for $i=0$.

\textbf{Induction step:} Suppose that $0<i\leq k$ and that (\ref{eq:imposc cnd2})
holds for $i-1$. We define
\[
v(x):=\tau^{-i}\gamma^{-i}u(\tau^{i}x).
\]
Then $v$ solves
\[
-\Delta v-(p(\tau^{i}x)-2)\frac{\left\langle D^{2}v(Dv+\gamma^{-i}q),Dv+\gamma^{i}q\right\rangle }{\left|Dv+\gamma^{-i}q\right|^{2}+(\gamma^{-i}\varepsilon)^{2}}=0\quad\text{in }B_{1}.
\]
Moreover, by induction hypothesis $\left|Dv(x)\right|=\gamma^{-i}\left|Du(\tau^{i}x)\right|\leq\gamma^{-i}\gamma^{i}=1$
in $B_{1}$. Therefore by Lemma \ref{lem:imposc dir} we have that
\begin{equation}
\left|\left\{ x\in B_{\tau}:Dv\cdot d\leq l\right\} \right|>\mu\left|B_{\tau}\right|\quad\text{implies}\quad Dv\cdot d\leq\gamma\text{ in }B_{\tau}\label{eq:imposc 1}
\end{equation}
whenever $d\in\partial B_{1}$. Since
\[
\left|\left\{ x\in B_{\tau}:Dv\cdot d\leq l\right\} \right|>\mu\left|B_{\tau}\right|\iff\left|\left\{ x\in B_{\tau^{i+1}}:Du\cdot d\leq l\gamma^{i}\right\} \right|>\mu\left|B_{\tau^{i+1}}\right|,
\]
we have by (\ref{eq:imposc cnd}) and (\ref{eq:imposc 1}) that $Dv\cdot d\leq\gamma$
in $B_{\tau}$. This implies that $Du\cdot d\leq\gamma^{i+1}$ in
$B_{\tau^{i+1}}$. Since $d\in\partial B_{1}$ was arbitrary, we obtain
(\ref{eq:imposc cnd2}) for $i$.
\end{proof}

\subsection{Hölder gradient estimates}

In this section we apply the improvement of oscillation to prove $C^{1,\alpha}$-estimates
for solutions to (\ref{eq:regularized homogeneous}). We need the
following regularity result by Savin \cite{savin07}.
\begin{lem}
\label{lem:small perturbation}Suppose that $p$ is smooth. Let $u$
be a smooth solution to (\ref{eq:regularized homogeneous}) in $B_{1}$
with $\left|Du\right|\leq1$ and either $q=0$ or $\left|q\right|>2$.
Then for any $\beta>0$ there exist positive constants $\eta(N,\hat{p},\beta)$
and $C(N,\hat{p},\beta)$ such that if 
\[
\left|u-L\right|\leq\eta\quad\text{in }B_{1}
\]
for some affine function $L$ satisfying $1/2\leq\left|DL\right|\leq1$,
then we have
\[
\left|Du(x)-Du(0)\right|\leq C\left|x\right|^{\beta}\quad\text{for all }x\in B_{1/2}.
\]
\end{lem}

\begin{proof}
Set $v:=u-L$. Then $v$ solves
\begin{equation}
-\Delta u-\frac{(p(x)-2)\left\langle D^{2}u(Du+q+DL),Du+q+DL\right\rangle }{\left|Du+q+DL\right|^{2}+\varepsilon^{2}}=0\quad\text{in }B_{1}.\label{eq:small perturbation}
\end{equation}
Observe that by the assumption $1/2\leq\left|DL\right|\leq1$ we have
$\left|Du+q+DL\right|\geq1/4$ if $\left|Du\right|\leq1/4$. It therefore
follows from \cite[Theorem 1.3]{savin07} (see also \cite{wang13})
that $\left\Vert v\right\Vert _{C^{2,\beta}(B_{1/2})}\leq C$ which
implies the claim.
\end{proof}
We also use the following simple consequence of Morrey's inequality.
\begin{lem}
\label{lem:morrey lemma}Let $u:B_{1}\rightarrow\mathbb{R}$ be a
smooth function with $\left|Du\right|\leq1$. For any $\theta>0$
there are constants $\varepsilon_{1}(N,\theta),\varepsilon_{0}(N,\theta)<1$
such that if the condition
\[
\left|\left\{ x\in B_{1}:\left|Du-d\right|>\varepsilon_{0}\right\} \right|\leq\varepsilon_{1}
\]
is satisfied for some $d\in S^{N-1}$, then there is $a\in\mathbb{R}$
such that
\[
\left|u(x)-a-d\cdot x\right|\leq\theta\text{ for all }x\in B_{1/2}.
\]
\end{lem}

\begin{proof}
By Morrey's inequality (see for example\ \cite[Theorem 4.10]{measuretheoryevans})
\begin{align*}
\underset{x\in B_{1/2}}{\osc}(u(x)-d\cdot x) & =\sup_{x,y\in B_{1/2}}\left|u(x)-d\cdot x-u(y)+d\cdot y\right|\\
 & \leq C(N)\Big(\int_{B_{1}}\left|Du-d\right|^{2N}\d x\Big)^{\frac{1}{2N}}\\
 & \leq C(N)(\varepsilon_{1}^{\frac{1}{2N}}+\varepsilon_{0}).
\end{align*}
Therefore, denoting $a:=\inf_{x\in B_{1/2}}(u(x)-d\cdot x)$, we have
for any $x\in B_{1/2}$
\[
\left|u(x)-a-d\cdot x\right|\leq\osc_{B_{1/2}}(u(x)-d\cdot x)\leq C(N)(\varepsilon_{1}^{\frac{1}{2N}}+\varepsilon_{0})\leq\theta,
\]
where the last inequality follows by taking small enough $\varepsilon_{0}$
and $\varepsilon_{1}$.
\end{proof}
We are now ready to prove a Hölder estimate for the gradient of solutions
to (\ref{eq:regularized homogeneous}). We first restrict the range
of $\left|q\right|$.
\begin{lem}
\label{thm:regularized apriori} Suppose that $p$ is smooth. Let
$u$ be a smooth solution to (\ref{eq:regularized homogeneous}) in
$B_{1}$ with $\left|Du\right|\leq1$ and either $q=0$ or $\left|q\right|>2$.
Then there exists a constant $\alpha(N,\hat{p})\in(0,1)$ such that
\[
\left\Vert Du\right\Vert _{C^{\alpha}(B_{1/2})}\leq C(N,\hat{p}).
\]
\end{lem}

\begin{proof}
For $\beta=1/2,$ let $\eta>0$ be as in Lemma \ref{lem:small perturbation}.
For $\theta=\eta/2$, let $\varepsilon_{0},\varepsilon_{1}$ be as
in Lemma \ref{lem:morrey lemma}. Set
\[
l:=1-\frac{\varepsilon_{0}^{2}}{2}\quad\text{and}\quad\mu:=\frac{\varepsilon_{1}}{\left|B_{1}\right|}.
\]
For these $l$ and $\mu$, let $\tau,\gamma\in(0,1)$ be as in Theorem
\ref{thm:imposc}. Let $k\geq0$ be the minimum integer such that
the condition (\ref{eq:imposc cnd}) does not hold. 

\textbf{Case $k=\infty$:} Theorem \ref{thm:imposc} implies that
\[
\left|Du\right|\leq\gamma^{i+1}\quad\text{in }B_{\tau^{i+1}}\text{ for all }i\geq0.
\]
Let $x\in B_{\tau}\setminus\left\{ 0\right\} $. Then $\tau^{i+1}\leq\left|x\right|\leq\tau^{i}$
for some $i\geq0$ and therefore
\[
i\leq\frac{\log\left|x\right|}{\log\tau}\leq i+1.
\]
We obtain
\begin{equation}
\left|Du(x)\right|\leq\gamma^{i}=\frac{1}{\gamma}\gamma^{i+1}\leq\frac{1}{\gamma}\gamma^{\frac{\log\left|x\right|}{\log\tau}}=\frac{1}{\gamma}\gamma^{\frac{\log\left|x\right|}{\log\gamma}\cdot\frac{\log\gamma}{\log\tau}}=:C\left|x\right|^{\alpha},\label{eq:regularized apriori 0}
\end{equation}
where $C=1/\gamma$ and $\alpha=\log\gamma/\log\tau$.

\textbf{Case $k<\infty$:} There is $d\in\partial B_{1}$ such that
\begin{equation}
\left|\left\{ x\in B_{\tau^{k+1}}:Du\cdot d\leq l\gamma^{k}\right\} \right|\leq\mu\left|B_{\tau^{k+1}}\right|.\label{eq:regularized apriori 1}
\end{equation}
We set 
\[
v(x):=\tau^{-k-1}\gamma^{-k}u(\tau^{k+1}x).
\]
Then $v$ solves
\[
-\Delta v-(p(\tau^{k+1}x)-2)\frac{\left\langle D^{2}v(Dv+\gamma^{-k}q),Dv+\gamma^{-k}q\right\rangle }{\left|Dv+\gamma^{-k}q\right|^{2}+\gamma^{-2k}\varepsilon^{2}}=0\quad\text{in }B_{1}
\]
and by (\ref{eq:regularized apriori 1}) we have
\begin{align}
\left|\left\{ x\in B_{1}:Dv\cdot d\leq l\right\} \right| & =\left|\left\{ x\in B_{1}:Du(\tau^{k+1}x)\cdot d\leq l\gamma^{k}\right\} \right|\nonumber \\
 & =\tau^{-N(k+1)}\left|\left\{ x\in B_{\tau^{k+1}}:Du(x)\cdot d\leq l\gamma^{k}\right\} \right|\nonumber \\
 & \leq\tau^{-N(k+1)}\mu\left|B_{\tau^{k+1}}\right|=\mu\left|B_{1}\right|=\varepsilon_{1}.\label{eq:regularized apriori 2}
\end{align}
Since either $k=0$ or (\ref{eq:imposc cnd}) holds for $k-1$, it
follows from Theorem \ref{thm:imposc} that $\left|Du\right|\leq\gamma^{k}$
in $B_{\tau^{k}}$. Thus
\begin{equation}
\left|Dv(x)\right|=\gamma^{-k}\left|Du(\tau^{k+1}x)\right|\leq1\quad\text{in }B_{1}.\label{eq:regularized apriori 3}
\end{equation}
For vectors $\xi,d\in B_{1}$, it is easy to verify the following
fact
\[
\left|\xi-d\right|>\varepsilon_{0}\implies\xi\cdot d\leq1-\varepsilon_{0}^{2}/2=l.
\]
Therefore, in view of (\ref{eq:regularized apriori 2}) and (\ref{eq:regularized apriori 3}),
we obtain
\[
\left|\left\{ x\in B_{1}:\left|Dv-d\right|>\varepsilon_{0}\right\} \right|\leq\varepsilon_{1}.
\]
Thus by Lemma \ref{lem:morrey lemma} there is $a\in\mathbb{R}$ such
that
\[
\left|v(x)-a-d\cdot x\right|\leq\theta=\eta/2\quad\text{for all }x\in B_{1/2}.
\]
Consequently, by applying Lemma \ref{lem:small perturbation} on the
function $2v(2^{-1}x)$, we find a positive constant $C(N,\hat{p})$
and $e\in\partial B_{1}$ such that
\[
\left|Dv(x)-e\right|\leq C\left|x\right|\quad\text{in }B_{1/4}.
\]
Since $\left|Dv\right|\leq1$, we have also
\[
\left|Dv(x)-e\right|\leq C\left|x\right|\quad\text{in }B_{1}.
\]
Recalling the definition of $v$ and taking $\alpha^{\prime}\in(0,1)$
so small that $\gamma/\tau^{\alpha^{\prime}}<1$ we obtain
\begin{equation}
\left|Du(x)-\gamma^{k}e\right|\leq C\gamma^{k}\tau^{-k-1}\left|x\right|\leq\frac{C}{\tau^{\alpha^{\prime}}}\left(\frac{\gamma}{\tau^{\alpha^{\prime}}}\right)^{k}\left|x\right|^{\alpha^{\prime}}\leq C\left|x\right|^{\alpha^{\prime}}\quad\text{in }B_{\tau^{k+1}},\label{eq:regularized apriori 4}
\end{equation}
where we absorbed $\tau^{\alpha^{\prime}}$ into the constant. On
the other hand, we have
\[
\left|Du\right|\leq\gamma^{i+1}\quad\text{in }B_{\tau^{i+1}}\text{ for all }i=0,\ldots,k-1
\]
so that, if $\tau^{i+2}\leq\left|x\right|\leq\tau^{i+1}$ for some
$i\in\left\{ 0,\ldots,k-1\right\} $, it holds that
\[
\left|Du(x)-\gamma^{k}e\right|\leq2\gamma^{i+1}\frac{\left|x\right|^{\alpha^{\prime}}}{\left|x\right|^{\alpha^{\prime}}}\leq\frac{2}{\tau^{\alpha^{\prime}}}\left(\frac{\gamma}{\tau^{\alpha^{\prime}}}\right)^{i+1}\left|x\right|^{\alpha^{\prime}}\le C\left|x\right|^{\alpha^{\prime}}.
\]
Combining this with (\ref{eq:regularized apriori 4}) we obtain
\begin{equation}
\left|Du(x)-\gamma^{k}e\right|\leq C\left|x\right|^{\alpha^{\prime}}\quad\text{in }B_{\tau}.\label{eq:regularized apriori 5}
\end{equation}

The claim now follows from (\ref{eq:regularized apriori 0}) and (\ref{eq:regularized apriori 5})
by standard translation arguments. \begin{details}

Let $x_{0}\in B_{1}$. We set
\[
v(x):=2u(\frac{1}{2}(x-x_{0})).
\]
Since $v$ solves
\[
\Delta v+(p(\frac{1}{2}(x-x_{0}))-2)\frac{\left\langle D^{2}v(Dv+q),Dv+q\right\rangle }{\left|Dv+q\right|^{2}+\varepsilon^{2}}=0\quad\text{in }B_{1},
\]
by Lemma \ref{thm:regularized apriori} there exists $C(p,N)$ and
$\alpha(p,N)$ such that 
\[
\left|Dv(x)-Dv(0)\right|\leq C\left|x\right|^{\alpha}\quad\text{for all }x\in B_{1/2}.
\]
In other words, we have
\[
\left|Du(\frac{1}{2}(x-x_{0}))-Du(\frac{1}{2}x_{0})\right|\leq C\left|x\right|^{\alpha}\quad\text{for all }x\in B_{1/2},x_{0}\in B_{1},
\]
from which it follows that
\[
\left|Du(x-x_{0})-Du(x_{0})\right|\le C\left|x\right|^{\alpha}\quad\text{for all }x\in B_{1/4},x_{0}\in B_{1/2}.
\]
This implies the Hölder estimate
\[
\left\Vert Du\right\Vert _{C^{\alpha}(B_{1/2})}\leq C.
\]
\end{details}
\end{proof}
\begin{thm}
\label{cor:h=0000F6lder estimate for regularized} Let $u$ be a bounded
viscosity solution to (\ref{eq:regularized homogeneous}) in $B_{1}$
with $q\in\mathbb{R}^{N}$. Then 
\begin{equation}
\left\Vert u\right\Vert _{C^{1,\alpha}(B_{1/2})}\leq C(N,\hat{p},\left\Vert u\right\Vert _{L^{\infty}(B_{1})})\label{eq:h=0000F6lder estimate for regularized 1}
\end{equation}
for some $\alpha(N,\hat{p})\in(0,1)$.
\end{thm}

\begin{proof}
Suppose first that $p$ is smooth. Let $\nu_{0}(N,\hat{p},\left\Vert u\right\Vert _{L^{\infty}(B_{1})})$
and $C_{0}(N,\hat{p},\left\Vert u\right\Vert _{L^{\infty}(B_{1})})$
be as in the Lipschitz estimate (Theorem \ref{thm:Lipschitz estimate}
in the Appendix) and set
\[
M:=2\max(\nu_{0},C_{0}).
\]

If $\left|q\right|>M$, then by Theorem \ref{lem:lipschitz lemma}
we have
\[
\left|Du\right|\leq C_{0}\quad\text{in }B_{1/2}.
\]
We set $\tilde{u}(x):=2u(x/2)/C_{0}$. Then $\left|D\tilde{u}\right|\leq1$
in $B_{1}$ and $\tilde{u}$ solves 
\[
-\Delta\tilde{u}-(p(x/2)-2)\frac{\left\langle D^{2}\tilde{u}(D\tilde{u}+q/C_{0}),D\tilde{u}+q/C_{0}\right\rangle }{\left|D\tilde{u}+q/C_{0}\right|^{2}+(\varepsilon/C_{0})^{2}}=0\quad\text{in }B_{1},
\]
where $q/C_{0}>2$. Thus by Theorem \ref{thm:regularized apriori}
we have
\[
\left\Vert D\tilde{u}\right\Vert _{C^{\alpha}(B_{1/2})}\leq C(N,\hat{p}),
\]
which implies (\ref{eq:h=0000F6lder estimate for regularized 1})
by standard translation arguments.

If $\left|q\right|\leq M$, we define
\[
w:=u-q\cdot x.
\]
Then by Theorem \ref{thm:Lipschitz estimate} we have
\[
\left|Dw\right|\leq C(N,\hat{p},\left\Vert w\right\Vert _{L^{\infty}(B_{1})})=:C^{\prime}(N,\hat{p},\left\Vert u\right\Vert _{L^{\infty}(B_{1})})\quad\text{in }B_{1/2}.
\]
We set $\tilde{w}(x):=2w(x/2)/C^{\prime}.$ Then $\left|D\tilde{w}\right|\leq1$
and so by Theorem \ref{lem:small perturbation} we have
\[
\left\Vert D\tilde{w}\right\Vert _{C^{\alpha}(B_{1/2})}\leq C(N,\hat{p}),
\]
which again implies (\ref{eq:h=0000F6lder estimate for regularized 1}).

Suppose then that $p$ is merely Lipschitz continuous. Take a sequence
$p_{j}\in C^{\infty}(B_{1})$ such that $p_{j}\rightarrow p$ uniformly
in $B_{1}$ and $\left\Vert Dp_{j}\right\Vert _{L^{\infty}(B_{1})}\leq\left\Vert Dp\right\Vert _{L^{\infty}(B_{1})}$.
For $r<1$, let $u_{j}$ be a solution to the Dirichlet problem
\[
\begin{cases}
-\Delta u_{j}-(p_{j}(x)-2)\frac{\left\langle D^{2}u(Du_{j}+q),Du_{j}+q\right\rangle }{\left|Du_{j}+q\right|^{2}+\varepsilon^{2}}=0 & \text{in }B_{r},\\
u_{j}=u & \text{on }B_{r}.
\end{cases}
\]
As observed in Proposition \ref{prop:c infty}, the solution exists
and we have $u_{j}\in C^{\infty}(B_{r})$. By comparison principle
$\left\Vert u_{j}\right\Vert _{L^{\infty}(B_{r})}\leq\left\Vert u\right\Vert _{L^{\infty}(B_{1})}$.
Then by the first part of the proof we have the estimate 
\[
\left\Vert u_{j}\right\Vert _{C^{1,\beta}(B_{r/2})}\leq C(N,\hat{p},\left\Vert u\right\Vert _{L^{\infty}(B_{1})}).
\]
By \cite[Theorem 4.14]{caffarelliCabre} the functions $u_{j}$ are
equicontinuous in $B_{1}$ and so by the Ascoli-Arzela theorem we
have $u_{j}\rightarrow v$ uniformly in $B_{1}$ up to a subsequence.
Moreover, by the stability principle $v$ is a solution to (\ref{eq:regularized homogeneous})
in $B_{r}$ and thus by comparison principle \cite[Theorem 2.6]{kawohlKutev07}
we have $v\equiv u$. By extracting a further subsequence, we may
ensure that also $Du_{j}\rightarrow Du$ uniformly in $B_{r/2}$ and
so the estimate $\left\Vert Du\right\Vert _{C^{1,\beta}(B_{r/2})}\leq C(N,\hat{p},\left\Vert u\right\Vert _{L^{\infty}(B_{1})})$
follows.
\end{proof}

\section{Hölder gradient estimates for the regularized inhomogeneous equation\label{sec:sec 2}}

In this section we consider the inhomogeneous equation 
\begin{equation}
-\Delta u-(p(x)-2)\frac{\left\langle D^{2}u(Du+q),Du+q\right\rangle }{\left|Du\right|^{2}+\varepsilon^{2}}=f(x)\quad\text{in }B_{1},\label{eq:non-homogeneous reguralized}
\end{equation}
where $p:B_{1}\rightarrow\mathbb{R}$ is Lipschitz continuous, $p_{\min}>1$,
$\varepsilon>0$, $q\in\mathbb{R}^{N}$ and $f\in C(B_{1})$ is bounded.
We apply the $C^{1,\alpha}$-estimates obtained in Theorem \ref{cor:h=0000F6lder estimate for regularized}
to prove regularity estimates for solutions of (\ref{eq:non-homogeneous reguralized})
with $q=0$. Our arguments are similar to those in \cite[Section 3]{attouchiParv},
see also \cite{imbertSilvestre12}. The idea is to use the well known
characterization of $C^{1,\alpha}$-regularity via affine approximates.
The following lemma plays a key role: It states that if $f$ is small,
then a solution to (\ref{eq:non-homogeneous reguralized}) can be
approximated by an affine function. This combined with scaling properties
of the equation essentially yields the desired affine functions.
\begin{lem}
\label{lem:non-homogeneous regularized first lemma}There exist constants
$\epsilon(N,\hat{p})$,$\tau(N,\hat{p})\in(0,1)$ such that the following
holds: If $\left\Vert f\right\Vert _{L^{\infty}(B_{1})}\leq\epsilon$
and $w$ is a viscosity solution to (\ref{eq:non-homogeneous reguralized})
in $B_{1}$ with $q\in\mathbb{R}^{N}$, $w(0)=0$ and $\osc_{B_{1}}w\leq1$,
then there exists $q^{\prime}\in\mathbb{R}^{N}$ such that
\[
\osc_{B_{\tau}}(w(x)-q^{\prime}\cdot x)\leq\frac{1}{2}\tau.
\]
 Moreover, we have $\left|q^{\prime}\right|\le C(N,\hat{p})$.
\end{lem}

\begin{proof}
Suppose on the contrary that the claim does not hold. Then, for a
fixed $\tau(N,\hat{p})$ that we will specify later, there exists
a sequence of Lipschitz continuous functions $p_{j}:B_{1}\rightarrow\mathbb{R}$
such that 
\[
p_{\min}\leq\inf_{B_{1}}p_{j}\leq\sup_{B_{1}}p_{j}\leq p_{\max}\quad\text{and}\quad(p_{j})_{L}\leq p_{L},
\]
 functions $f_{j}\in C(B_{1})$ such that $f_{j}\rightarrow0$ uniformly
in $B_{1}$, vectors $q_{j}\in\mathbb{R}^{N}$ and viscosity solutions
$w_{j}$ to 
\[
-\Delta w_{j}-(p_{j}(x)-2)\frac{\left\langle D^{2}w_{j}(Dw_{j}+q_{j}),Dw_{j}+q_{j}\right\rangle }{\left|Dw_{j}+q_{j}\right|^{2}+\varepsilon^{2}}=f_{j}(x)\quad\text{in }B_{1}
\]
such that $w_{j}(0)=0$, $\osc_{B_{1}}w_{j}\leq1$ and 
\begin{equation}
\osc_{B_{\tau}}(w_{j}(x)-q^{\prime}\cdot x)>\frac{\tau}{2}\quad\text{for all }q^{\prime}\in\mathbb{R}^{N}.\label{eq:first 0}
\end{equation}
By \cite[Proposition 4.10]{caffarelliCabre}, the functions $w_{j}$
are uniformly Hölder continuous in $B_{r}$ for any $r\in(0,1)$.
Therefore by the Ascoli-Arzela theorem, we may extract a subsequence
such that $w_{j}\rightarrow w_{\infty}$ and $p_{j}\rightarrow p_{\infty}$
uniformly in $B_{r}$ for any $r\in(0,1)$. Moreover, $p_{\infty}$
is $p_{L}$-Lipschitz continuous and $p_{\min}\leq p_{\infty}\leq p_{\max}$.
It then follows from (\ref{eq:first 0}) that

\begin{equation}
\osc_{B_{\tau}}(w_{\infty}(x)-q^{\prime}\cdot x)>\frac{\tau}{2}\quad\text{for all }q^{\prime}\in\mathbb{R}^{N}.\label{eq:first 1}
\end{equation}
We have two cases: either $q_{j}$ is bounded or unbounded.

\textbf{Case $q_{j}$ is bounded: }In this case $q_{j}\rightarrow q_{\infty}\in\mathbb{R}^{N}$
up to a subsequence. It follows from the stability principle that
$w_{\infty}$ is a viscosity solution to
\begin{equation}
-\Delta w_{\infty}-(p_{\infty}(x)-2)\frac{\left\langle D^{2}w_{\infty}(Dw_{\infty}+q_{\infty}),Dw_{\infty}+q_{\infty}\right\rangle }{\left|Dw_{\infty}+q_{\infty}\right|^{2}+\varepsilon^{2}}=0\quad\text{in }B_{1}.\label{eq:first -1}
\end{equation}
Hence by Theorem \ref{cor:h=0000F6lder estimate for regularized}
we have $\left\Vert Dw_{\infty}\right\Vert _{C^{\beta_{1}}(B_{1/2})}\leq C(N,\hat{p})$
for some $\beta_{1}(N,\hat{p})$. The mean value theorem then implies
the existence of $q^{\prime}\in\mathbb{R}^{N}$ such that
\[
\osc_{B_{r}}(u-q^{\prime}\cdot x)\leq C_{1}(N,\hat{p})r^{1+\beta_{1}}\quad\text{for all }r\leq\frac{1}{2}.
\]

\textbf{Case $q_{j}$ is unbounded:} In this case we take a subsequence
such that $\left|q_{j}\right|\rightarrow\infty$ and the sequence
$d_{j}:=d_{j}/\left|d_{j}\right|$ converges to $d_{\infty}\in\partial B_{1}$.
Then $w_{j}$ is a viscosity solution to
\[
-\Delta w_{j}-(p_{j}(x)-2)\frac{\left\langle D^{2}w_{j}(\sa q_{j}\sa^{-1}Dw_{j}+d_{j}),\sa q_{j}\sa^{-1}Dw_{j}+d_{j}\right\rangle }{\left|\left|q_{j}\right|^{-1}Dw_{j}+d_{j}\right|^{2}+\left|q_{j}\right|^{-2}\varepsilon^{2}}=f_{j}(x)\quad\text{in }B_{1}.
\]
It follows from the stability principle that $w_{\infty}$ is a viscosity
solution to
\[
-\Delta w_{j}-(p_{\infty}(x)-2)\left\langle D^{2}w_{\infty}d_{\infty},d_{\infty}\right\rangle =0\quad\text{in }B_{1}.
\]
By \cite[Theorem 8.3]{caffarelliCabre} there exist positive constants
$\beta_{2}(N,\hat{p})$, $C_{2}(N,\hat{p})$, $r_{2}(N,\hat{p})$
and a vector $q^{\prime}\in\mathbb{R}^{N}$ such that 
\[
\osc_{B_{r}}(w_{\infty}-q^{\prime}\cdot x)\leq C_{2}r^{1+\beta_{2}}\quad\text{for all }r\leq r_{2}.
\]

We set $C_{0}:=\max(C_{1},C_{2})$ and $\beta_{0}:=\min(\beta_{1},\beta_{2})$.
Then by the two different cases there always exists a vector $q^{\prime}\in\mathbb{R}^{N}$
such that
\[
\osc_{B_{r}}(w_{\infty}-q^{\prime}\cdot x)\leq C_{0}r^{1+\beta_{0}}\quad\text{for all }r\leq\min(\frac{1}{2},r_{2}).
\]
We take $\tau$ so small that $C_{0}\tau^{\beta_{0}}\leq\frac{1}{4}$
and $\tau\leq\min(\frac{1}{2},r_{2})$. Then, by substituting $r=\tau$
in the above display, we obtain
\begin{equation}
\osc_{B_{\tau}}(w_{\infty}-q^{\prime}\cdot x)\leq C_{0}\tau^{\beta_{0}}\tau\leq\frac{1}{4}\tau,\label{eq:first 2}
\end{equation}
which contradicts (\ref{eq:first 1}).

The bound $\left|q^{\prime}\right|\le C(N,\hat{p})$ follows by observing
that (\ref{eq:first 2}) together with the assumption $\osc_{B_{1}}w\leq1$
yields $\left|q^{\prime}\right|\leq C$. Thus the contradiction is
still there even if (\ref{eq:first 1}) is weakened by requiring additionally
that $\left|q^{\prime}\right|\leq C$.
\end{proof}
\begin{lem}
\label{lem:second lemma} Let $\tau(N,\hat{p})$ and $\epsilon(N,\hat{p})$
be as in Lemma \ref{lem:non-homogeneous regularized first lemma}.
If $\left\Vert f\right\Vert _{L^{\infty}(B_{1})}\leq\epsilon$ and
$u$ is a viscosity solution to (\ref{eq:non-homogeneous reguralized})
in $B_{1}$ with $q=0$, $u(0)=0$ and $\osc_{B_{1}}u\leq1$, then
there exists $\alpha\in(0,1)$ and $q_{\infty}\in\mathbb{R}^{N}$
such that
\[
\sup_{B_{\tau^{k}}}\left|u(x)-q_{\infty}\cdot x\right|\leq C(N,\hat{p})\tau^{k(1+\alpha)}\quad\text{for all }k\in\mathbb{N}.
\]
\end{lem}

\begin{proof}
\textbf{Step 1:} We show that there exists a sequence $(q_{k})_{k=0}^{\infty}\subset\mathbb{R}^{N}$
such that 
\begin{equation}
\osc_{B_{\tau^{k}}}(u(x)-q_{k}\cdot x)\leq\tau^{k(1+\alpha)}.\label{eq:second lemma 1}
\end{equation}
When $k=0$, this estimate holds by setting $q_{0}=0$ since $u(0)=0$
and $\osc_{B_{1}}\leq1$. Next we take $\alpha\in(0,1)$ such that
$\tau^{\alpha}>\frac{1}{2}$. We assume that $k\ge0$ and that we
have already constructed $q_{k}$ for which (\ref{eq:second lemma 1})
holds. We set
\[
w_{k}(x):=\tau^{-k(1+\alpha)}(u(\tau^{k}x)-q_{k}\cdot(\tau^{k}x))
\]
and
\[
f_{k}(x):=\tau^{k(1-\alpha)}f(\tau^{k}x).
\]
Then by induction assumption $\osc_{B_{1}}(w_{k})\leq1$ and $w_{k}$
is a viscosity solution to
\[
-\Delta w_{k}-\frac{(p(\tau^{k}x)-2)\left\langle D^{2}w_{k}(Dw_{k}+\tau^{-k\alpha}q_{k}),Dw_{k}+\tau^{-k\alpha}q_{k}\right\rangle }{\left|Dw_{k}+\tau^{-k\alpha}q_{k}\right|^{2}+(\tau^{-k\alpha}\varepsilon)^{2}}=f_{k}(x)\quad\text{in }B_{1}.
\]
By Lemma \ref{lem:non-homogeneous regularized first lemma} there
exists $q_{k}^{\prime}\in\mathbb{R}^{N}$ with $\left|q_{k}^{\prime}\right|\leq C(N,\hat{p})$
such that 
\[
\osc_{B_{\tau}}(w_{k}(x)-q_{k}^{\prime}\cdot x)\leq\frac{1}{2}\tau.
\]
Using the definition of $w_{k}$, scaling to $B_{\tau^{k+1}}$ and
dividing by $\tau^{-k(\alpha+1)}$, we obtain from the above
\[
\osc_{B_{\tau^{k+1}}}(u(x)-(q_{k}+\tau^{k\alpha}q_{k}^{\prime})\cdot x)\leq\frac{1}{2}\tau^{1+k(1+\alpha)}\leq\tau^{(k+1)(1+\alpha)}.
\]
Denoting $q_{k+1}:=q_{k}+\tau^{k\alpha}q_{k}^{\prime}$, the above
estimate is condition (\ref{eq:second lemma 1}) for $k+1$ and the
induction step is complete.

\textbf{Step 2:} Observe that whenever $m>k$, we have
\[
\left|q_{m}-q_{k}\right|\leq\sum_{i=k}^{m-1}\tau^{i\alpha}\left|q_{i}^{\prime}\right|\leq C(N,\hat{p})\sum_{i=k}^{m-1}\tau^{i\alpha}.
\]
Therefore $q_{k}$ is a Cauchy sequence and converges to some $q_{\infty}\in\mathbb{R}^{N}$.
Thus
\[
\sup_{x\in B_{\tau^{k}}}(q_{k}\cdot x-q_{\infty}\cdot x)\leq\tau^{k}\left|q_{k}-q_{\infty}\right|\leq\tau^{k}\sum_{i=k}^{\infty}\tau^{i\alpha}q_{i}^{\prime}\leq C(N,\hat{p})\tau^{k(1+\alpha)}.
\]
This with (\ref{eq:second lemma 1}) implies that
\[
\sup_{x\in B_{\tau^{k}}}\left|u(x)-q_{\infty}\cdot x\right|\leq C(N,\hat{p})\tau^{k(1+\alpha)}.\qedhere
\]
\end{proof}
\begin{thm}
\label{cor:main corollary}Suppose that $u$ is a viscosity solution
to (\ref{eq:non-homogeneous reguralized}) in $B_{1}$ with $q=0$
and $\osc_{B_{1}}\leq1$. Then there are constants $\alpha(N,\hat{p})$
and $C(N,\hat{p},\left\Vert f\right\Vert _{L^{\infty}(B_{1})})$ such
that 
\[
\left\Vert u\right\Vert _{C^{1,\alpha}(B_{1/2})}\leq C.
\]
\end{thm}

\begin{proof}
Let $\epsilon(N,\hat{p})$ and $\tau(N,\hat{p})$ be as in Lemma \ref{lem:second lemma}.
Set
\[
v(x):=\kappa u(x/4)
\]
where $\kappa:=\epsilon(1+\left\Vert f\right\Vert _{L^{\infty}(B_{1})})^{-1}$.
For $x_{0}\in B_{1}$, set
\[
w(x):=v(x+x_{0})-v(x_{0}).
\]
Then $\osc_{B_{1}}w\leq1$, $w(0)=0$ and $w$ is a viscosity solution
to
\[
-\Delta w-\frac{(p(x/4+x_{0}/4)-2)\left\langle D^{2}wDw,Dw\right\rangle }{\left|Dw\right|^{2}+\varepsilon^{2}\kappa^{2}/4^{2}}=g(x)\quad\text{in }B_{1},
\]
where $g(x):=\kappa f(x/4+x_{0}/4)/4^{2}$. Now $\left\Vert g\right\Vert _{L^{\infty}(B_{1})}\leq\epsilon$
so by Lemma \ref{lem:second lemma} there exists $q_{\infty}(x_{0})\in\mathbb{R}^{N}$
such that
\[
\sup_{x\in B_{\tau^{k}}}\left|w(x)-q_{\infty}(x_{0})\cdot x\right|\leq C(N,\hat{p})\tau^{k(1+\alpha)}\quad\text{for all }k\in\mathbb{N}.
\]
Thus we have shown that for any $x_{0}\in B_{1}$ there exists a vector
$q_{\infty}(x_{0})$ such that
\[
\sup_{x\in B_{r}(x_{0})}\left|v(x)-v(x_{0})-q_{\infty}(x_{0})\cdot(x-x_{0})\right|\leq C(N,\hat{p})r{}^{1+\alpha}\quad\text{for all }r\in(0,1].
\]
This together with a standard argument (see for example \cite[Lemma A.1]{attouchiParv})
implies that $[Dv]_{C^{\alpha}(B_{1})}\leq C(N,\hat{p})$ and so by
defintion of $v$, also $[Du]_{C^{\alpha}(B_{1/4})}\leq C(N,\hat{p},\left\Vert f\right\Vert _{L^{\infty}(B_{1})})$.
The conclusion of the theorem then follows by a standard translation
argument.
\end{proof}

\section{Proof of the main theorem}

In this section we finish the proof our main theorem.
\begin{proof}[Proof of Theorem \ref{thm:main-1}]
We may assume that $u\in C(\overline{B}_{1})$. By Comparison Principle
(Lemma \ref{lem:comparison principle} in the Appendix) $u$ is the
unique viscosity solution to
\begin{equation}
\begin{cases}
-\Delta v-\frac{(p(x)-2)\left\langle D^{2}vDv,Dv\right\rangle }{\left|Dv\right|^{2}}=f(x)+u-v & \text{in }B_{1},\\
v=u & \text{on }\partial B_{1}.
\end{cases}\label{eq:main 1}
\end{equation}
By \cite[Theorem 15.18]{gilbargTrudinger01} there exists a classical
solution $u_{\varepsilon}$ to the approximate problem
\[
\begin{cases}
-\Delta u_{\varepsilon}-\frac{(p_{\varepsilon}(x)-2)\left\langle D^{2}u_{\varepsilon}Du_{\varepsilon},Du_{\varepsilon}\right\rangle }{\left|Du_{\varepsilon}\right|^{2}+\varepsilon^{2}}=f_{\varepsilon}(x)+u-u_{\varepsilon} & \text{in }B_{1},\\
u_{\varepsilon}=u & \text{on }\partial B_{1},
\end{cases}
\]
where $p_{\varepsilon},f_{\varepsilon},u_{\varepsilon}\in C^{\infty}(B_{1})$
are such that $p_{\varepsilon}\rightarrow p$, $f_{\varepsilon}\rightarrow f$
and $u_{\varepsilon}\rightarrow u_{0}$ uniformly in $B_{1}$ as $\varepsilon\rightarrow0$
and $\left\Vert Dp_{\varepsilon}\right\Vert _{L^{\infty}(B_{1})}\leq\left\Vert Dp\right\Vert _{L^{\infty}(B_{1})}$.
The maximum principle implies that $\left\Vert u_{\varepsilon}\right\Vert _{L^{\infty}(B_{1})}\leq2\left\Vert f\right\Vert _{L^{\infty}(B_{1})}+2\left\Vert u\right\Vert _{L^{\infty}(B_{1})}$.
By \cite[Proposition 4.14]{caffarelliCabre} the solutions $u_{\varepsilon}$
are equicontinuous in $\overline{B}_{1}$ (their modulus of continuity
depends only on $N$, $p$, $\left\Vert f\right\Vert _{L^{\infty}(B_{1})}$,
$\left\Vert u\right\Vert _{L^{\infty}(B_{1})}$ and modulus of continuity
of $u$). Therefore by the Ascoli-Arzela theorem we have $u_{\varepsilon}\rightarrow v\in C(\overline{B}_{1})$
uniformly in $\overline{B}_{1}$ up to a subsequence. By the stability
principle, $v$ is a viscosity solution to (\ref{eq:main 1}) and
thus by uniqueness $v\equiv u$. 

By Corollary \ref{cor:main corollary} we have $\alpha(N,\hat{p})$
such that
\begin{equation}
\left\Vert Du_{\varepsilon}\right\Vert _{C^{\alpha}(B_{1/2})}\leq C(N,\hat{p},\left\Vert f\right\Vert _{L^{\infty}(B_{1})},\left\Vert u\right\Vert _{L^{\infty}(B_{1})})\label{eq:main 2}
\end{equation}
and by the Lipschitz estimate \ref{thm:Lipschitz estimate} also
\[
\left\Vert Du_{\varepsilon}\right\Vert _{L^{\infty}(B_{1/2})}\leq C(N,\hat{p},\left\Vert f\right\Vert _{L^{\infty}(B_{1})},\left\Vert u\right\Vert _{L^{\infty}(B_{1})}).
\]
Therefore by the Ascoli-Arzela theorem there exists a subsequence
such that $Du_{\varepsilon}\rightarrow\eta$ uniformly in $B_{1/2}$,
where the function $\eta:B_{1/2}\rightarrow\mathbb{R}^{N}$ satisfies
\[
\left\Vert \eta\right\Vert _{C^{\alpha}(B_{1/2})}\leq C(N,\hat{p},\left\Vert f\right\Vert _{L^{\infty}(B_{1})},\left\Vert u\right\Vert _{L^{\infty}(B_{1})}).
\]
Using the mean value theorem and the estimate (\ref{eq:main 2}),
we deduce for all $x,y\in B_{1/2}$
\begin{align*}
 & \left|u(y)-u(x)-(y-x)\cdot\eta(x)\right|\\
 & \ \leq\left|u_{\varepsilon}(x)-u_{\varepsilon}(y)-(y-x)\cdot Du_{\varepsilon}(x)\right|\\
 & \ \ \ \ +\left|u(y)-u_{\varepsilon}(y)-u(x)+u_{\varepsilon}(x)\right|+\left|x-y\right|\left|\eta(x)-Du_{\varepsilon}(x)\right|\\
 & \leq C(N,\hat{p},\left\Vert u\right\Vert _{L^{\infty}(B_{1})})\left|x-y\right|^{1+\alpha}+o(\varepsilon)/\varepsilon.
\end{align*}
Letting $\varepsilon\rightarrow0$, this implies that $Du(x)=\eta(x)$
for all $x\in B_{1/2}$.
\end{proof}

\appendix

\section{Lipschitz estimate}

In this section we apply the method of Ishii and Lions \cite{ishiiLions90}
to prove a Lipschitz estimate for solutions to the inhomogeneous normalized
$p(x)$-Laplace equation and its regularized or perturbed versions.
We need the following vector inequality.
\begin{lem}
\label{lem:lipschitz lemma}Let $a,b\in\mathbb{R}^{N}\setminus\left\{ 0\right\} $
with $a\not=b$ and $\varepsilon\geq0$. Then
\[
\left|\frac{a}{\sqrt{\left|a\right|^{2}+\varepsilon^{2}}}-\frac{b}{\sqrt{\left|b\right|^{2}+\varepsilon^{2}}}\right|\leq\frac{2}{\max\left(\left|a\right|,\left|b\right|\right)}\left|a-b\right|.
\]
\end{lem}

\begin{proof}
We may suppose that $\left|a\right|=\max(\left|a\right|,\left|b\right|)$.
Let $s_{1}:=\sqrt{\left|a\right|^{2}+\varepsilon^{2}}$ and $s_{2}:=\sqrt{\left|b\right|^{2}+\varepsilon^{2}}$.
Then
\begin{align*}
\left|\frac{a}{s_{1}}-\frac{b}{s_{2}}\right|=\frac{1}{s_{1}}\left|a-b+\frac{b}{s_{2}}(s_{2}-s_{1})\right| & \leq\frac{1}{s_{1}}(\left|a-b\right|+\frac{\left|b\right|}{s_{2}}\left|s_{2}-s_{1}\right|)\\
 & \leq\frac{1}{\left|a\right|}(\left|a-b\right|+\left|s_{2}-s_{1}\right|).
\end{align*}
Moreover
\begin{align*}
\left|s_{2}-s_{1}\right| & =\left|\sqrt{\left|a\right|^{2}+\varepsilon^{2}}-\sqrt{\left|b\right|^{2}+\varepsilon^{2}}\right|=\frac{\left|\left|a\right|^{2}-\left|b\right|^{2}\right|}{\sqrt{\left|a\right|^{2}+\varepsilon^{2}}+\sqrt{\left|b\right|^{2}+\varepsilon^{2}}}\\
 & \leq\frac{(\left|a\right|+\left|b\right|)\left|\left|a\right|-\left|b\right|\right|}{\left|a\right|+\left|b\right|}\leq\left|a-b\right|\qedhere.
\end{align*}
\end{proof}
$ $
\begin{thm}[Lipschitz estimate]
\label{thm:Lipschitz estimate} Suppose that $p:B_{1}\rightarrow\mathbb{R}$
is Lipschitz continuous, $p_{\min}>1$ and that $f\in C(B_{1})$ is
bounded. Let $u$ be a viscosity solution to
\[
-\Delta u-(p(x)-2)\frac{\left\langle D^{2}u(Du+q),Du+q\right\rangle }{\left|Du+q\right|^{2}+\varepsilon^{2}}=f(x)\quad\text{in }B_{1},
\]
where $\varepsilon\geq0$ and $q\in\mathbb{R}^{N}$. Then there are
constants $C_{0}(N,\hat{p},\left\Vert u\right\Vert _{L^{\infty}(B_{1})},\left\Vert f\right\Vert _{L^{\infty}(B_{1})})$
and $\nu_{0}(N,\hat{p})$ such that if $\left|q\right|>\nu_{0}$ or
$\left|q\right|=0$, then we have
\[
\left|u(x)-u(y)\right|\leq C_{0}\left|x-y\right|\quad\text{for all }x,y\in B_{1/2}.
\]
\end{thm}

\begin{proof}
We let $r(N,\hat{p})\in(0,1/2)$ denote a small constant that will
be specified later. Let $x_{0},y_{0}\in B_{r/2}$ and define the function
\[
\Psi(x,y):=u(x)-u(y)-L\varphi(\left|x-y\right|)-\frac{M}{2}\left|x-x_{0}\right|^{2}-\frac{M}{2}\left|y-y_{0}\right|^{2},
\]
where $\varphi:[0,2]\rightarrow\mathbb{R}$ is given by
\[
\varphi(s):=s-s^{\gamma}\kappa_{0},\quad\kappa_{0}:=\frac{1}{\gamma2^{\gamma+1}},
\]
and the constants $L(N,\hat{p},\left\Vert u\right\Vert _{L^{\infty}(B_{1})}),M(N,\hat{p},\left\Vert u\right\Vert _{L^{\infty}(B_{1})})>0$
and $\gamma(N,\hat{p})\in(1,2)$ are also specified later. Our objective
is to show that for a suitable choice of these constants, the function
$\Psi$ is non-positive in $\overline{B_{r}}\times\overline{B_{r}}$.
By the definition of $\varphi$, this yields $u(x_{0})-u(y_{0})\leq L\left|x_{0}-y_{0}\right|$
which implies that $u$ is $L$-Lipschitz in $B_{r}$. The claim of
the theorem then follows by standard translation arguments.

Suppose on contrary that $\Psi$ has a positive maximum at some point
$(\hat{x},\hat{y})\in\overline{B_{r}}\times\overline{B_{r}}$. Then
$\hat{x}\not=\hat{y}$ since otherwise the maximum would be non-positive.
We have
\begin{align}
0 & <u(\hat{x})-u(\hat{y})-L\varphi(\left|\hat{x}-\hat{y}\right|)-\frac{M}{2}\left|\hat{x}-x_{0}\right|^{2}-\frac{M}{2}\left|\hat{y}-y_{0}\right|^{2}\nonumber \\
 & \leq\left|u(\hat{x})-u(\hat{y})\right|-\frac{M}{2}\left|\hat{x}-x_{0}\right|^{2}.\label{eq:lipschitz est 1}
\end{align}
Therefore, by taking 
\begin{equation}
M:=\frac{8\osc_{B_{1}}u}{r^{2}},\label{eq:lipschitz est M}
\end{equation}
 we get
\[
\left|\hat{x}-x_{0}\right|\leq\sqrt{\frac{2}{M}\left|u(\hat{x})-u(\hat{y})\right|}\leq r/2
\]
and similarly $\left|\hat{y}-y_{0}\right|\leq r/2$. Since $x_{0},y_{0}\in B_{r/2}$,
this implies that $\hat{x},\hat{y}\in B_{r}$. 

By \cite[Proposition 4.10]{caffarelliCabre} there exist constants
$C^{\prime}(N,\hat{p},\left\Vert u\right\Vert _{L^{\infty}(B_{1})},\left\Vert f\right\Vert _{L^{\infty}(B_{1})})$
and $\beta(N,\hat{p})\in(0,1)$ such that 
\begin{equation}
\left|u(x)-u(y)\right|\leq C^{\prime}\left|x-y\right|^{\beta}\quad\text{for all }x,y\in B_{r}.\label{eq:lipschitz est 2}
\end{equation}
It follows from (\ref{eq:lipschitz est 1}) and (\ref{eq:lipschitz est 2})
that for $C_{0}:=\sqrt{2C^{\prime}}\sqrt{M}$ we have
\begin{align}
M\left|\hat{x}-x_{0}\right|\leq C_{0}\left|\hat{x}-\hat{y}\right|^{\beta/2},\nonumber \\
M\left|\hat{y}-y_{0}\right|\leq C_{0}\left|\hat{x}-\hat{y}\right|^{\beta/2}.\label{eq:lipschitz est 3}
\end{align}
Since $\hat{x}\not=\hat{y}$, the function $(x,y)\mapsto\varphi(\left|x-y\right|)$
is $C^{2}$ in a neighborhood of $(\hat{x},\hat{y})$ and we may invoke
the Theorem of sums \cite[Theorem 3.2]{userguide}. For any $\mu>0$
there exist matrices $X,Y\in S^{N}$ such that
\begin{align*}
(D_{x}(L\varphi(\left|x-y\right|))(\hat{x},\hat{y}),X) & \in\overline{J}^{2,+}(u-\frac{M}{2}\left|x-x_{0}\right|^{2})(\hat{x}),\\
(-D_{y}(L\varphi(\left|x-y\right|))(\hat{x},\hat{y}),Y) & \in\overline{J}^{2,-}(u+\frac{M}{2}\left|y-y_{0}\right|^{2})(\hat{y}),
\end{align*}
which by denoting $z:=\hat{x}-\hat{y}$ and
\begin{align*}
a & :=L\varphi^{\prime}(\left|z\right|)\frac{z}{\left|z\right|}+M(\hat{x}-x_{0}),\\
b & :=L\varphi^{\prime}(\left|z\right|)\frac{z}{\left|z\right|}-M(\hat{y}-y_{0}),
\end{align*}
 can be written as
\begin{equation}
(a,X+MI)\in\overline{J}^{2,+}u(\hat{x}),\quad(b,Y-MI)\in\overline{J}^{2,-}u(\hat{y}).\label{eq:lipschitz est 6}
\end{equation}
By assuming that $L$ is large enough depending on $C_{0}$, we have
by (\ref{eq:lipschitz est 3}) and the fact $\varphi^{\prime}\in\left[\frac{3}{4},1\right]$
\begin{align}
\left|a\right|,\left|b\right| & \leq L\left|\varphi^{\prime}(\left|\hat{x}-\hat{y}\right|)\right|+C_{0}\left|\hat{x}-\hat{y}\right|^{\beta/2}\leq2L,\label{eq:lipschitz est 4}\\
\left|a\right|,\left|b\right| & \geq L\left|\varphi^{\prime}(\left|\hat{x}-\hat{y}\right|)\right|-C_{0}\left|\hat{x}-\hat{y}\right|^{\beta/2}\geq\frac{1}{2}L.\label{eq:lipschitz est 5}
\end{align}
Moreover, we have
\begin{align}
-(\mu+2\left\Vert B\right\Vert )\begin{pmatrix}I & 0\\
0 & I
\end{pmatrix} & \leq\begin{pmatrix}X & 0\\
0 & -Y
\end{pmatrix}\nonumber \\
 & \leq\begin{pmatrix}B & -B\\
-B & B
\end{pmatrix}+\frac{2}{\mu}\begin{pmatrix}B^{2} & -B^{2}\\
-B^{2} & B^{2}
\end{pmatrix},\label{eq:lipschitz est 7}
\end{align}
where
\begin{align*}
B & =L\varphi^{\prime\prime}(\left|z\right|)\frac{z}{\left|z\right|}\otimes\frac{z}{\left|z\right|}+\frac{L\varphi^{\prime}(\left|z\right|)}{\left|z\right|}\left(I-\frac{z}{\left|z\right|}\otimes\frac{z}{\left|z\right|}\right),\\
B^{2} & =BB=L^{2}(\varphi^{\prime\prime}(\left|z\right|))^{2}\frac{z}{\left|z\right|}\otimes\frac{z}{\left|z\right|}+\frac{L^{2}(\varphi^{\prime}(\left|z\right|))^{2}}{\left|z\right|^{2}}\left(I-\frac{z}{\left|z\right|}\otimes\frac{z}{\left|z\right|}\right).
\end{align*}
Using that $\varphi^{\prime\prime}(\left|z\right|)<0<\varphi^{\prime}(\left|z\right|)$
and $\left|\varphi^{\prime\prime}(\left|z\right|)\right|\leq\varphi^{\prime}(\left|z\right|)/\left|z\right|$,
we deduce that
\begin{equation}
\left\Vert B\right\Vert \leq\frac{L\varphi^{\prime}(\left|z\right|)}{\left|z\right|}\quad\text{and}\quad\text{\ensuremath{\left\Vert B^{2}\right\Vert \leq\frac{L^{2}(\varphi^{\prime}(\left|z\right|))^{2}}{\left|z\right|^{2}}}}.\label{eq:lipschitz b est}
\end{equation}
Moreover, choosing
\[
\mu:=4L\left(\left|\varphi^{\prime\prime}(\left|z\right|)\right|+\frac{\left|\varphi^{\prime}(\left|z\right|)\right|}{\left|z\right|}\right),
\]
and using that $\varphi^{\prime\prime}(\left|z\right|)<0$, we have
\begin{align}
\left\langle B\frac{z}{\left|z\right|},\frac{z}{\left|z\right|}\right\rangle +\frac{2}{\mu}\left\langle B^{2}\frac{z}{\left|z\right|},\frac{z}{\left|z\right|}\right\rangle =L\varphi^{\prime\prime}(\left|z\right|)+\frac{2}{\mu}L^{2}\left|\varphi^{\prime\prime}(\left|z\right|)\right| & \leq\frac{L}{2}\varphi^{\prime\prime}(\left|z\right|).\label{eq:lipschitz est 8}
\end{align}
We set $\eta_{1}:=a+q$ and $\eta_{2}:=b+q$. By (\ref{eq:lipschitz est 4})
and (\ref{eq:lipschitz est 5}) there is a constant $\nu_{0}(L)$
such that if $\left|q\right|=0$ or $\left|q\right|>\nu_{0}$, then
\begin{equation}
\left|\eta_{1}\right|,\left|\eta_{2}\right|\geq\frac{L}{2}.\label{eq:lipschitz est 9}
\end{equation}
We denote $A(x,\eta):=I+(p(x)-2)\eta\otimes\eta$ and $\overline{\eta}:=\frac{\eta}{\sqrt{\left|\eta\right|^{2}+\varepsilon^{2}}}$.
Since $u$ is a viscosity solution, we obtain from (\ref{eq:lipschitz est 6})
\begin{align}
0 & \leq\tr(A(\hat{x},\overline{\eta}_{1})(X+MI))-\tr(A(\hat{y},\overline{\eta}_{2})(Y-MI))+f(\hat{x})-f(\hat{y})\nonumber \\
 & =\tr(A(\hat{y},\overline{\eta}_{2})(X-Y))+\tr((A(\hat{x},\overline{\eta}_{2})-A(\hat{y},\overline{\eta}_{2}))X)\nonumber \\
 & \ \ \ +\tr((A(\hat{x},\overline{\eta}_{1})-A(\hat{x},\overline{\eta}_{2}))X)+M\tr(A(\hat{x},\overline{\eta}_{1})+A(\hat{y},\overline{\eta}_{2}))\nonumber \\
 & \ \ \ +f(\hat{x})-f(\hat{y})\nonumber \\
 & =:T_{1}+T_{2}+T_{3}+T_{4}+T_{5}.\label{eq:lipschitz est 11}
\end{align}
We will now proceed to estimate these terms. The plan is to obtain
a contradiction by absorbing the other terms into $T_{1}$ which is
negative by concavity of $\varphi$.

\textbf{Estimate of $T_{1}$: }Multiplying (\ref{eq:lipschitz est 7})
by the vector $(\frac{z}{\left|z\right|},-\frac{z}{\left|z\right|})$
and using (\ref{eq:lipschitz est 8}), we obtain an estimate for the
smallest eigenvalue of $X-Y$
\begin{align*}
\lambda_{\min}(X-Y) & \leq\left\langle (X-Y)\frac{z}{\left|z\right|},\frac{z}{\left|z\right|}\right\rangle \\
 & \leq4\left\langle B\frac{z}{\left|z\right|},\frac{z}{\left|z\right|}\right\rangle +\frac{8}{\mu}\left\langle B^{2}\frac{z}{\left|z\right|},\frac{z}{\left|z\right|}\right\rangle \leq2L\varphi^{\prime\prime}(\left|z\right|).
\end{align*}
The eigenvalues of $A(\hat{y},\overline{\eta}_{2})$ are between $\min(1,\pmin-1)$
and $\max(1,p_{\max}-1)$. Therefore by \cite{theobald75}
\begin{align*}
T_{1}=\tr(A(\hat{y},\overline{\eta}_{2})(X-Y)) & \leq\sum_{i}\lambda_{i}(A(\hat{y},\overline{\eta}_{2}))\lambda_{i}(X-Y)\\
 & \leq\min(1,\pmin-1)\lambda_{\min}(X-Y)\\
 & \leq C(\hat{p})L\varphi^{\prime\prime}(\left|z\right|).
\end{align*}

\textbf{Estimate of $T_{2}$: }We have
\[
T_{2}=\tr((A(\hat{x},\overline{\eta}_{2})-A(\hat{y},\overline{\eta}_{2}))X)\leq\left|p(\hat{x})-p(\hat{y})\right|\left|\left\langle X\overline{\eta}_{2},\overline{\eta}_{2}\right\rangle \right|\leq C(\hat{p})\left|z\right|\left\Vert X\right\Vert ,
\]
where by (\ref{eq:lipschitz est 7}) and (\ref{eq:lipschitz b est})
\begin{align}
\left\Vert X\right\Vert \leq\left\Vert B\right\Vert +\frac{2}{\mu}\left\Vert B\right\Vert ^{2} & \leq\frac{L\left|\varphi^{\prime}(\left|z\right|)\right|}{\left|z\right|}+\frac{2L^{2}(\varphi^{\prime}(\left|z\right|))^{2}}{4L(\left|\varphi^{\prime\prime}(\left|z\right|)\right|+\frac{\left|\varphi^{\prime}(\left|z\right|)\right|}{\left|z\right|})\left|z\right|^{2}}\nonumber \\
 & \leq\frac{2L\varphi^{\prime}(\left|z\right|)}{\left|z\right|}.\label{eq:lipschitz est 12}
\end{align}

\textbf{Estimate of $T_{3}$: }From Lemma \ref{lem:lipschitz lemma}
and the estimate (\ref{eq:lipschitz est 9}) it follows that
\begin{align}
\left|\overline{\eta}_{1}-\overline{\eta}_{2}\right| & \leq\frac{2\left|\eta_{1}-\eta_{2}\right|}{\max(\left|\eta_{1}\right|,\left|\eta_{2}\right|)}\leq\frac{4}{L}\left|\eta_{1}-\eta_{2}\right|=\frac{4}{L}\left|a-b\right|\nonumber \\
 & \leq\frac{4}{L}(M\left|\hat{x}-x_{0}\right|+M\left|\hat{y}-y_{0}\right|)\leq\frac{8C_{0}}{L}\left|z\right|^{\beta/2},\label{eq:lipschitz est 10}
\end{align}
where in the last inequality we used (\ref{eq:lipschitz est 3}).
Observe that
\[
\left\Vert \overline{\eta}_{1}\otimes\overline{\eta}_{1}-\overline{\eta}_{2}\otimes\overline{\eta}_{2}\right\Vert =\left\Vert (\overline{\eta}_{1}-\overline{\eta}_{2})\otimes\overline{\eta}_{1}-\overline{\eta}_{2}\otimes(\overline{\eta}_{2}-\overline{\eta}_{1})\right\Vert \leq(\left|\overline{\eta}_{1}\right|+\left|\overline{\eta}_{2}\right|)\left|\overline{\eta}_{1}-\overline{\eta}_{2}\right|.
\]
Using the last two displays, we obtain by \cite{theobald75} and (\ref{eq:lipschitz est 12})
\begin{align*}
T_{3}=\tr((A(\hat{x},\overline{\eta}_{1})-A(\hat{x},\overline{\eta}_{2}))X) & \leq N\left\Vert A(x_{1},\overline{\eta}_{1})-A(x_{1},\overline{\eta}_{2})\right\Vert \left\Vert X\right\Vert \\
 & \leq N\left|p(x_{1})-2\right|(\left|\overline{\eta}_{1}\right|+\left|\overline{\eta}_{2}\right|)\left|\overline{\eta}_{1}-\overline{\eta}_{2}\right|\left\Vert X\right\Vert \\
 & \leq\frac{C(N,\hat{p})C_{0}}{L}\left|z\right|^{\beta/2}\left\Vert X\right\Vert \\
 & \leq C(N,\hat{p},\left\Vert u\right\Vert _{L^{\infty}},\left\Vert f\right\Vert _{L^{\infty}})\sqrt{M}\varphi^{\prime}(\left|z\right|)\left|z\right|^{\beta/2-1}.
\end{align*}

\textbf{Estimate of $T_{4}$ and $T_{5}$: }By Lipschitz continuity
of $p$ we have
\begin{align*}
T_{4} & =M\tr(A(\hat{x},\overline{\eta}_{1})+A(\hat{y},\overline{\eta}_{2}))\leq2MC(N,\hat{p}).
\end{align*}
We have also
\[
T_{5}=f(\hat{x})-f(\hat{y})\leq2\left\Vert f\right\Vert _{L^{\infty}(B_{1})}.
\]

Combining the estimates, we deduce the existence of positive constants
$C_{1}(N,\hat{p})$ and $C_{2}(N,\hat{p},\left\Vert u\right\Vert _{L^{\infty}(B_{1})},\left\Vert f\right\Vert _{L^{\infty}(B_{1})})$
such that
\begin{align}
0 & \leq C_{1}L\varphi^{\prime\prime}(\left|z\right|)+C_{2}\big(L\varphi^{\prime}(\left|z\right|)+\sqrt{M}\varphi^{\prime}(\left|z\right|)\left|z\right|^{\frac{\beta}{2}-1}+M+1\big)\nonumber \\
 & \leq C_{1}L\varphi^{\prime\prime}(\left|z\right|)+C_{2}(L+\sqrt{M}\left|z\right|^{\frac{\beta}{2}-1}+M+1)\label{eq:lipschitz est 15}
\end{align}
where we used that $\varphi^{\prime}(\left|z\right|)\in[\frac{3}{4},1]$.
We take $\gamma:=\frac{\beta}{2}+1$ so that we have
\[
\varphi^{\prime\prime}(\left|z\right|)=\frac{1-\gamma}{2^{\gamma+1}}\left|z\right|^{\gamma-2}=\frac{-\beta}{2^{\frac{\beta}{2}+3}}\left|z\right|^{\frac{\beta}{2}-1}=:-C_{3}\left|z\right|^{\frac{\beta}{2}-1}.
\]
We apply this to (\ref{eq:lipschitz est 15}) and obtain
\begin{align}
0 & \leq(C_{2}\sqrt{M}-C_{1}C_{3}L)\left|z\right|^{\frac{\beta}{2}-1}+C_{2}(L+M+1)\label{eq:lipschitz est 155}
\end{align}
We fix $r:=\frac{1}{2}\left(\frac{6C_{2}}{C_{1}C_{3}}\right)^{\frac{1}{\frac{\beta}{2}-1}}$.
By (\ref{eq:lipschitz est M}) this will also fix $M=(N,\hat{p},\left\Vert u\right\Vert _{L^{\infty}(B_{1})})$.
We take $L$ so large that 
\[
L>\max(\frac{2C_{2}\sqrt{M}}{C_{1}C_{3}},M+1).
\]
Then by (\ref{eq:lipschitz est 155}) we have
\begin{align*}
0<-\frac{1}{2}C_{1}C_{3}L\left|z\right|^{\frac{\beta}{2}-1}+2C_{2}L & \leq L(-\frac{1}{2}C_{1}C_{3}(2r)^{\frac{\beta}{2}-1}+2C_{2})\\
 & =-LC_{2}\leq0,
\end{align*}
which is a contradiction.
\end{proof}

\section{Stability and comparison principles}
\begin{lem}
Suppose that $p\in C(B_{1})$, $p_{\min}>1$ and that $f:B_{1}\times\mathbb{R}\rightarrow\mathbb{R}$
is continuous. Let $u_{\varepsilon}$ be a viscosity solution to
\[
-\Delta u_{\varepsilon}-(p_{\varepsilon}(x)-2)\frac{\left\langle D^{2}u_{\varepsilon}Du_{\varepsilon},Du_{\varepsilon}\right\rangle }{\left|Du_{\varepsilon}\right|^{2}+\varepsilon^{2}}=f_{\varepsilon}(x,u(x))\quad\text{in }B_{1}
\]
and assume that $u_{\varepsilon}\rightarrow u\in C(B_{1})$, $p_{\varepsilon}\rightarrow p$
and $f_{\varepsilon}\rightarrow f$ locally uniformly as $\varepsilon\rightarrow0$.
Then $u$ is a viscosity solution to
\[
-\Delta u-(p(x)-2)\frac{\left\langle D^{2}uDu,Du\right\rangle }{\left|Du\right|^{2}}=f(x,u(x))\quad\text{in }B_{1}.
\]
\end{lem}

\begin{proof}
It is enough to consider supersolutions. Suppose that $\varphi\in C^{2}$
touches $u$ from below at $x$. Since $u_{\varepsilon}\rightarrow u$
locally uniformly, there exists a sequence $x_{\varepsilon}\rightarrow x$
such that $u_{\varepsilon}-\varphi$ has a local minimum at $x_{\varepsilon}$.
We denote $\eta_{\varepsilon}:=D\varphi(x_{\varepsilon})/\sqrt{\left|D\varphi(x_{\varepsilon})\right|^{2}+\varepsilon^{2}}.$
Then $\eta_{\varepsilon}\rightarrow\eta\in\overline{B}_{1}$ up to
a subsequence. Therefore we have
\begin{align}
0 & \leq-\Delta\varphi(x_{\varepsilon})-(p_{\varepsilon}(x_{\varepsilon})-2)\left\langle D^{2}\varphi(x_{\varepsilon})\eta_{\varepsilon},\eta_{\varepsilon}\right\rangle -f_{\varepsilon}(x_{\varepsilon},u_{\varepsilon}(x_{\varepsilon}))\nonumber \\
 & \rightarrow-\Delta\varphi(x)-(p(x)-2)\left\langle D^{2}\varphi(x_{\varepsilon})\eta,\eta\right\rangle -f(x,u(x)),\label{eq:stability 1}
\end{align}
which is what is required in Definition \ref{def:viscosity solutions}
in the case $D\varphi(x)=0$. If $D\varphi(x)\not=0$, then $D\varphi(x_{\varepsilon})\not=0$
when $\varepsilon$ is small and thus $\eta=D\varphi(x)/\left|D\varphi(x)\right|$.
Therefore \ref{eq:stability 1} again implies the desired inequality.
\end{proof}
\begin{lem}
\label{lem:comparison principle}Suppose that $p:B_{1}\rightarrow\mathbb{R}$
is Lipschitz continuous, $p_{\min}>1$ and that $f\in C(B_{1})$ is
bounded. Assume that $u\in C(\overline{B}_{1})$ is a viscosity subsolution
to $-\Delta_{p(x)}^{N}u\leq f-u$ in $B_{1}$ and that $v\in C(\overline{B}_{1})$
is a viscosity supersolution to $-\Delta_{p(x)}^{N}v\geq f-v$ in
$B_{1}$. Then
\[
u\leq v\quad\text{on }\partial B_{1}
\]
implies
\[
u\leq v\quad\text{in }B_{1}.
\]
\end{lem}

\begin{proof}
\textbf{Step 1:} Assume on the contrary that the maximum of $u-v$
in $\overline{B}_{1}$ is positive. For $x,y\in\overline{B}_{1}$,
set
\[
\Psi_{j}(x,y):=u(x)-v(y)-\varphi_{j}(x,y),
\]
where $\varphi_{j}(x,y):=\frac{j}{4}\left|x-y\right|^{4}$. Let $(x_{j},y_{j})$
be a global maximum point of $\Psi_{j}$ in $\overline{B}_{1}\times\overline{B}_{1}$.
Then
\[
u(x_{j})-v(y_{j})-\frac{j}{4}\left|x_{j}-y_{j}\right|^{4}\geq u(0)-v(0)
\]
so that
\[
\frac{j}{4}\left|x_{j}-y_{j}\right|^{4}\leq2\left\Vert u\right\Vert _{L^{\infty}(B_{1})}+2\left\Vert v\right\Vert _{L^{\infty}(B_{1})}<\infty.
\]
By compactness and the assumption $u\leq v$ on $\partial B_{1}$
there exists a subsequence such that $x_{j},y_{j}\rightarrow\hat{x}\in B_{1}$
and $u(\hat{x})-v(\hat{x})>0$. Finally, since $(x_{j},y_{j})$ is
a maximum point of $\Psi_{j}$, we have
\[
u(x_{j})-v(x_{j})\leq u(x_{j})-v(y_{j})-\frac{j}{4}\left|x_{j}-y_{j}\right|^{4},
\]
and hence by continuity
\begin{equation}
\frac{j}{4}\left|x_{j}-y_{j}\right|^{4}\leq v(x_{j})-v(y_{j})\rightarrow0\label{eq:comparison convergence}
\end{equation}
as $j\rightarrow\infty$.

\textbf{Step 2:} If $x_{j}=y_{j}$, then $D_{x}^{2}\varphi_{j}(x_{j},y_{j})=D_{y}^{2}\varphi_{j}(x_{j},y_{j})=0$.
Therefore, since the function $x\mapsto u(x)-\varphi_{j}(x,y_{j})$
reaches its maximum at $x_{j}$ and $y\mapsto v(y)-(-\varphi_{j}(x_{j},y))$
reaches its minimum at $y_{j}$, we obtain from the definition of
viscosity sub- and supersolutions that 
\[
0\leq f(x_{j})-u(x_{j})\quad\text{and}\quad0\geq f(y_{j})-v(y_{j}).
\]
That is $0\leq f(x_{j})-f(y_{j})+v(y_{j})-u(x_{j}),$ which leads
to a contradiction since $x_{j},y_{j}\rightarrow\hat{x}$ and $v(\hat{x})-u(\hat{x})<0$.
We conclude that $x_{j}\not=y_{j}$ for all large $j$. Next we apply
the Theorem of sums \cite[Theorem 3.2]{userguide} to obtain matrices
$X,Y\in S^{N}$ such that
\[
(D_{x}\varphi(x_{j},y_{j}),X)\in\overline{J}^{2,+}u(x_{j}),\quad(-D_{y}\varphi(x_{j},y_{j}),Y)\in\overline{J}^{2,-}v(y_{j})
\]
 and
\begin{equation}
\begin{pmatrix}X & 0\\
0 & -Y
\end{pmatrix}\leq D^{2}\varphi(x_{j},y_{j})+\frac{1}{j}(D^{2}(x_{j},y_{j}))^{2},\label{eq:matrix ineq}
\end{equation}
where 
\[
D^{2}(x_{j},y_{j})=\begin{pmatrix}M & -M\\
-M & M
\end{pmatrix}
\]
 with $M=j(2(x_{j}-y_{j})\otimes(x_{j}-y_{j})+\left|x_{j}-y_{j}\right|^{2}I)$.
Multiplying the matrix inequality (\ref{eq:matrix ineq}) by the $\mathbb{R}^{2N}$
vector $(\xi_{1},\xi_{2})$ yields
\begin{align*}
\left\langle X\xi_{1},\xi_{1}\right\rangle -\left\langle Y\xi_{2},\xi_{2}\right\rangle  & \leq\left\langle (M+2j^{-1}M^{2})(\xi_{1}-\xi_{2}),\xi_{1}-\xi_{2}\right\rangle \\
 & \leq(\left\Vert M\right\Vert +2j^{-1}\left\Vert M\right\Vert ^{2})\left|\xi_{1}-\xi_{2}\right|^{2}.
\end{align*}
Observe also that $\eta:=D_{x}\varphi(x_{j},y_{j})=-D_{y}(x_{j},y_{j})=j\left|x_{j}-y_{j}\right|^{2}(x_{j}-y_{j})\not=0$
for all large $j$. Since $u$ is a subsolution and $v$ is a supersolution,
we thus obtain
\begin{align*}
 & f(y_{j})-f(x_{j})+u(x_{j})-v(y_{j})\\
 & \ \leq\tr(X-Y)+(p(x_{j})-2)\left\langle X\frac{\eta}{\left|\eta\right|},\frac{\eta}{\left|\eta\right|}\right\rangle -(p(y_{j})-2)\left\langle Y\frac{\eta}{\left|\eta\right|},\frac{\eta}{\left|\eta\right|}\right\rangle \\
 & \ \leq(p(x_{j})-1)\left\langle X\frac{\eta}{\left|\eta\right|},\frac{\eta}{\left|\eta\right|}\right\rangle -(p(y_{j})-1)\left\langle Y\frac{\eta}{\left|\eta\right|},\frac{\eta}{\left|\eta\right|}\right\rangle \\
 & \ \leq(\left\Vert M\right\Vert +2j^{-1}\left\Vert M\right\Vert ^{2})\Big|\sqrt{p(x_{j})-1}-\sqrt{p(y_{j})-1}\Big|^{2}\\
 & \ \leq Cj\left|x_{j}-y_{j}\right|^{2}\frac{\left|p(x_{j})-p(y_{j})\right|^{2}}{\left(\sqrt{p(x_{j})-1}+\sqrt{p(y_{j})-1}\right)^{2}}\\
 & \leq C(\hat{p})j\left|x_{j}-y_{j}\right|^{4}.
\end{align*}
This leads to a contradiction since the left-hand side tends to $u(\hat{x})-v(\hat{y})>0$
and the right-hand side tends to zero by (\ref{eq:comparison convergence}).

\bibliographystyle{alpha}

\add
\end{proof}

\end{document}